\pdfoutput=1
\documentclass{article} 

\usepackage{ssArxiv}
\usepackage{microtype}
\usepackage{graphicx}
\usepackage{subfigure}

\usepackage{color}
\usepackage{algorithm}
\usepackage{algorithmic}
\usepackage{natbib}
\usepackage{eso-pic}
\usepackage{forloop}
\usepackage[utf8]{inputenc}
\usepackage[T1]{fontenc} 
\usepackage{booktabs}       %
\usepackage{xcolor}         %
\usepackage{amsmath,amsbsy,amsgen,amscd,amssymb,amsthm,amsfonts}
\usepackage{mathtools}
\usepackage{bm}
\usepackage{enumitem}
\usepackage{parskip}

\usepackage{adjustbox}
\usepackage{hyperref}
\usepackage[nameinlink]{cleveref}
\newcommand{\norm}[1]{\|#1\|}
\newcommand{\R}{\mathbb{R}}

\newcommand{\expect}[1]{\mathbb{E}\left[#1\right]}
\newcommand{\expectk}[1]{\mathbb{E}_k\left[#1\right]}
\newcommand{\prox}{\mathrm{prox}}

\newcommand{\ip}[2]{\langle#1,#2\rangle}
\usepackage{setspace}
\usepackage{tikz}
\usepackage{pgfplots}
\usepgfplotslibrary{groupplots}
\usepgfplotslibrary{fillbetween}
\pgfplotsset{compat=1.18}
\pgfplotsset{
    myplotstyle/.style={
    ylabel style={align=center, font=\bfseries\boldmath},
    xlabel style={align=center, font=\bfseries\boldmath},
    x tick label style={font=\bfseries\boldmath},
    y tick label style={font=\bfseries\boldmath},
    scaled ticks=false,
    every axis plot/.append style={thick},
    xmin = 0 , xmax=100, ymin = 0, ymax = 3,
    },
}

\definecolor{mycolor1}{rgb}{0,0.44,0.74}%
\definecolor{mycolor2}{rgb}{0,0.59,0.2}%
\definecolor{mycolor3}{rgb}{0.85,0.325,0.098}%

\DeclareMathOperator*{\argmin}{argmin}
\DeclareMathOperator*{\argmax}{argmax}

\newtheorem{theorem}{Theorem}
\newtheorem{lemma}[theorem]{Lemma}

\newtheorem{corollary}[theorem]{Corollary}
\newtheorem{remark}[theorem]{Remark}
\newtheorem{proposition}[theorem]{Proposition}

\newtheorem{definition}{Definition}
\newtheorem*{proof-sketch}{Proof Sketch}
\usepackage{xspace}

\newcommand{\algo}{\textsc{Bdca}\xspace}
\newcommand{\algoem}{Block EM\xspace}

\definecolor{mycolor2}{rgb}{0,0.59,0.2}%

\newcommand{\gap}{\mathrm{gap}}

\hypersetup{
   colorlinks = true,
   citecolor  = darkblue,
   linkcolor  = darkred,
   filecolor  = darkblue,
   urlcolor   = darkblue,
 }

\title{{Randomized} Block Coordinate DC  {Algorithm}}
\author{\name{Hoomaan Maskan} 
\email{hoomaan.maskan@umu.se} \\
\addr{Ume\r{a} University, Sweden} \\[0.5em]
\name Paniz Halvachi 
\email{panizhalvachi@gmail.com} \\
\addr{} \\[0.5em]
\name{Suvrit Sra} 
\email{s.sra@tum.de} \\
\addr{Technical University of Munich, Germany} \\[0.5em]
\name{Alp Yurtsever} 
\email{alp.yurtsever@umu.se} \\
\addr{Ume\r{a} University, Sweden} \\[0.5em]
}

\begin{document}

\maketitle

\begin{abstract}
We introduce an extension of the Difference of Convex Algorithm (DCA) in the form of a  {randomized} block coordinate approach for problems with separable structure. 
For $n$ coordinate-blocks and $k$ iterations, our main result proves a non-asymptotic convergence rate of $O(n/k)$  {in expectation, with respect to a stationarity measure based on a Forward-Backward envelope}. 
Furthermore, leveraging the connection between DCA and Expectation Maximization (EM), we propose a  {randomized} block coordinate EM algorithm. 
\end{abstract}

\section{Introduction}

In nonconvex optimization, many problems find their abstraction through the lens of Difference of Convex (DC) programming, which involves minimizing an objective function represented as the difference of convex functions. 
This paper focuses on the following DC programming template:
\begin{equation}\label{eqn:main}
     \min_{\bm{x}\in \mathcal{M}} ~~ \phi(\bm{x}) := f(\bm{x}) + g(\bm{x}) - h(\bm{x}),
\end{equation}
where $f$, $g$, and $h$ are  {real-valued} lower semi-continuous convex functions defined on the closed convex set $\mathcal{M} \subseteq \mathbb{R}^m$,  { and throughout, we assume that a finite solution $\bm{x}_\star \in \mathcal{M}$ exists, satisfying $\phi_\star := \phi(\bm{x}_\star) \leq \phi(\bm{x})$ for all $\bm{x} \in \mathcal{M}$.}
We assume that the function $f$ is smooth  {(\textit{i.e.,} its gradient is Lipschitz continuous)}, and that the function $g$ and the set $\mathcal{M}$ exhibit a block-separable structure in relation to the coordinates of $\bm{x}$ (see \Cref{sec:BC-DC}). 
We do not impose the smoothness assumption on the terms $g$ and $h$. 

A large and popular class of algorithms for solving problem~\eqref{eqn:main}, introduced  {in} \cite{TAO1986249}, comprises instances of the DCA method. 
These algorithms rely on the simple yet effective idea of approximating  $\phi(\bm{x})$ with the convex surrogate function obtained by using the first-order Taylor approximation of the concave component $-h(\bm{x})$. 
Starting from a feasible initial point $\bm{x}^0 \in \mathcal{M}$, DCA solves a sequence of convex optimization subproblems defined as 
\begin{equation}\label{eqn:main-surrogate}
     \min_{\bm{x}\in \mathcal{M}} ~~ \hat{\phi}(\bm{x},\bm{x}^k) := f(\bm{x}) + g(\bm{x}) - h(\bm{x}^k) - \ip{ {\bm v^k}}{\bm{x} - \bm{x}^k},
\end{equation}
where $\bm{x}^k$ represents the current estimate  {and $\bm v^k\in\partial h(\bm{x}^k)$ denotes a subgradient of $h$ at $\bm{x}^k$.}

Observe that function $\hat{\phi}$ is a global upper bound on $\phi$ since $h$ is convex; as a result, DCA can be viewed as a variant of the Majorization–Minimization (MM) algorithm \cite{lange2016mm}.

The general template of DCA does not specify how, or to what degree of precision, we should solve these subproblems, and various instances can be derived by employing different techniques for their resolution. 
Motivated by its strong empirical performance, wide applicability, simplicity, and its interpretablity, DCA has been extensively studied. 

With the recent focus in machine learning on nonconvex optimization problems, along with the increasing complexity and scale of these problems, the demand for more efficient and scalable algorithms for nonconvex optimization continuously grows. 
A particular technique in addressing this challenge is the `coordinate descent' (CD) approach. 
The main idea in CD is to minimize the multivariate objective function by updating one variable (or a \emph{block} of variables) at every iteration, while keeping the others fixed. 
This approach often leads to simple subproblems that can be solved more efficiently, particularly for large-scale optimization problems commonly encountered in data science and related fields. 

DC programming and DCA have been widely used in machine learning applications. Examples include kernel selection \cite{argyriou2006dc}, sparse principal component analysis \cite{sriperumbudur2007sparse}, discrepancy estimation in domain adaptation \cite{awasthi2024best}, training neural networks \cite{awasthi2024dc}, and many more. 
However, to our knowledge, there is currently no true CD variant of DCA with convergence guarantees. 
Our main goal in this paper is to address this gap. 
With this in mind, we summarize now the key contributions of this paper: 
\begin{list}{{\color{darkred!90}\tiny$\blacksquare$}}{\leftmargin=2em}\vspace*{-5pt}
  \setlength{\itemsep}{2pt}
\item Our primary contribution is the development of a novel variant of the DCA method that incorporates randomized CD updates. 
We refer to this algorithm as the Block Coordinate DC algorithm (\algo). 
We analyze \algo and demonstrate that it converges  {in expectation} to a first-order stationary point of the problem at a non-asymptotic rate of $\mathcal{O}(n/k)$, for $n$ blocks and $k$ iterations. 
Importantly, our guarantees do not necessitate the functions $g$ or $h$ to be smooth. 
\item DCA relates to the well known Expectation Maximization (EM) algorithm when dealing with exponential family of distributions \cite{yuille2003concave,lethi30years}. 
Building on this connection, we introduce a  {randomized} block coordinate EM method, referred to as \algoem. %
\end{list}

\section{Related Work}

For a comprehensive survey on the variants and applications of DCA, we refer to \cite{lethi30years}. 
Here, we focus specifically on the variants involving block-coordinate or alternating updates. 

\subsection{Related Work on BDCA}

 {Pham Dinh et~al.}~\cite{pham2022alternating} recently proposed an alternating DC algorithm for partial DC problems
involving two variables, where the objective retains a DC form when one
variable is fixed. 
The algorithm is shown to converge to a Fréchet/Clarke critical point of the objective function under the Kurdyka-Łojasiewicz  {(KL)} property, with non-asymptotic convergence guarantees.

More recently,  {Yuan}~\cite{yuan2023coordinate} proposed a block-coordinate method for DC problems; however, this method differs from DCA, in that it solves a proximal subproblem instead of the DCA subproblem in  \eqref{eqn:main-surrogate}. 
Furthermore, their analysis is based on the so-called \emph{Luo-Tseng error bound} assumption, which posits that the (sub)gradients of the objective function converge to zero in norm continuously as they approach a stationary point. 
This assumption limits the applicability of these results in nonsmooth settings, as most nonsmooth functions of interest exhibit discontinuous changes in subgradient norm. 

DCA is also known as the convex-concave procedure (CCCP) when both the convex and concave terms are differentiable. 
Recently,  {Yurtsever and Sra}~\cite{yurtsever2022cccp} established
an equivalence between CCCP and the Frank-Wolfe algorithm applied to an
epigraph reformulation of the original DC problem.
As a result, they showed a $\mathcal{O}(1/k)$ non-asymptotic convergence rate for CCCP. 

Our analysis differs from  {theirs} in several key aspects. 
While they transfer guarantees from the Frank-Wolfe connection, we provide a direct proof. We consider the DCA template with potentially nonsmooth terms, which are managed through their subgradients. 
They demonstrate convergence in terms of 
\begin{equation*}
    \max_{\bm{x} \in \mathcal{M}} ~  {\Big\{} f(\bm{x}^k) - f(\bm{x}) - \ip{\nabla h(\bm{x}^k)}{\bm{x}^k-\bm{x}}  {\Big\}}
\end{equation*}
which is not amenable to our block-coordinate analysis unless $f$ is also separable. 
Consequently, we identify a new gap function (see \Cref{sec:BC-DC}) suitable for measuring proximity to first-order stationarity, which is an essential component of our analysis. 

\subsection{Related Work on \algoem}

 {Kumar and Schmidt}~\cite{kumar2017convergence} analyzed the convergence guarantees of the EM algorithm from a majorization-minimization perspective. 
Assuming that the surrogate optimization problems are strongly convex, they established that EM converges at a rate of $\mathcal{O}(1/k)$ in terms of the squared gradient norm of the negative log-likelihood. 
However, this strong convexity assumption may not hold in general, even for exponential family distributions. 

 {More recently, Kunstner et al.}~\cite{Kunstner2020HomeomorphicInvarianceOE} studied the non-asymptotic convergence of EM for exponential family distributions by drawing a connection to mirror descent, establishing a convergence rate of $\mathcal{O}(1/k)$ in Kullback-Leibler divergence and linking it to first-order stationarity through Bregman divergences.

\subsection{Other Work on Nonconvex BC Methods
} 
Block Successive Upperbound Minimization algorithm was proposed to minimize an approximated upperbound of any nonconvex and/or nonsmooth objective function  {by Razaviyayn et al.}~\cite{razaviyayn2013unified}. 
This method considers optimizing a sequence of approximate objective functions. 
They showed the asymptotic convergence of their  {method} for quasi-convex functions or objectives with compact level sets. 
Also, they assume that each subproblem of BSUM has a unique solution. 

 {Among} recent studies in nonconvex optimization using block-coordinate methods,  {Xu and Yin}~\cite{xu2017globally} proposed a block prox-linear method.
Assuming Lipschitz smoothness of the differentiable parts, they prove the convergence of a subsequence of their algorithm to a stationary point.
Further, assuming the KL condition, they show that the entire sequence converges to a critical point and estimated its asymptotic convergence rate. 

 {Aubry et al.~\cite{aubry2018new} proposed} a mixture of maximum block improvement and a sequential minimization method for continuously differentiable nonconvex problems,  {with asymptotic first-order optimality guarantees and an application to wireless sensor allocation problem.}

More recent studies targeted a nonasymptotic convergence analysis. 
In the realm of nonconvex composite optimization,  {Chorobura and Necoara}~\cite{chorobura2023random} considered objectives that
include two nonconvex nonseparable terms where one of them has Lipschitz gradients while the other function is continuously differentiable. 
They showed a convergence rate of $\mathcal{O}(n/\sqrt{k})$ in the sense of expected gradient norm when utilizing a randomized selection of the blocks. 
Their cyclic scheme achieves a rate of $\mathcal{O}(n^2/\sqrt{k})$ in the same sense.

 {Cai et al.~\cite{cai2023cyclic} recently proposed} a cyclic coordinate descent algorithm with variance reduction. 
Their objective considers a block separable and another separable sum function whose proximal operator is efficiently computable. 
Their analysis assumes a nonstandard Lipschitz gradient  {assumption defined with respect to block-dependent semi-norms} and achieves a rate of $\mathcal{O}(n/k)$ in the sense of distance of the objective's subdifferential to zero.

 {Deng and Lan studied a stochastic block mirror descent algorithm~\cite{dang2015stochastic} and, more recently, proposed a randomized coordinate proximal subgradient method~\cite{deng2020efficiency} applicable to nonconvex problem templates.
Gao et~al.~\cite{gao2020randomized} developed a randomized Bregman coordinate descent algorithm for nonconvex functions under relative smoothness. 
In a related direction, Faust et~al.~\cite{faust2023bregman} provided a Bregman divergence interpretation of DCA.}

 {Finally, Latafat et al.~\cite{latafat2022block} proposed a proximal coordinate method for nonconvex composite optimization. Their formulation complements ours in terms of separability assumptions. They assume the smooth component to be separable, whereas the nonsmooth term is not. However, unlike our method, which updates coordinates through local subproblems involving only $g_i$ for the selected coordinate blocks, their algorithm requires involves proximal mapping of the full $g$ before updating a randomly selected coordinate.
}

\section{Block Coordinate DCA}\label{sec:BC-DC} 

We begin by establishing notation and restating our assumptions for the model problem. 

We assume that the function $f$ is $L$-smooth, meaning that its gradient is Lipschitz continuous with a constant $L \geq 0$. 
In addition, we assume that the function $g$ and set $\mathcal{M}$ exhibit a block separable structure in relation to the coordinates of $\bm{x}$:
\begin{equation*}
     g(\bm{x}) = \sum_{i=1}^n g_i(\boldsymbol{D}_i \bm{x}),%
\end{equation*}
where $\boldsymbol{D}_i$ represents $(m_i \times m)$ dimensional row subsets of the identity matrix $\boldsymbol{I}$, serving as the selection operator for a non-overlapping partition of the coordinates into $n$ blocks: 
\begin{equation*}
     \sum_{i=1}^n \boldsymbol{D}_i^\top \boldsymbol{D}_i = \boldsymbol{I} \quad \text{and} \quad \sum_{i=1}^n m_i = m.
\end{equation*}
Similarly, $\mathcal{M}$ can be decomposed as $\mathcal{M}_1 \times \cdots \times \mathcal{M}_n$, where the components $\mathcal{M}_i \in \mathbb{R}^{m_i}$, such that 
\begin{equation*}
    \bm{x} \in \mathcal{M}
    \iff
    \boldsymbol{D}_i\bm{x} \in \mathcal{M}_i \quad \text{for} \quad i=1,\ldots,n.
\end{equation*}
For the ease of presentation, we define the following coordinate vector notation:
\begin{equation*}
    x_i = \boldsymbol{D}_i \bm{x}, 
    ~~
    \bm{x}_i = \boldsymbol{D}_i^\top \boldsymbol{D}_i \bm{x}, 
    ~~ \text{and} ~~
    \bar{\bm{x}}_i = (\boldsymbol{I} - \boldsymbol{D}_i^\top \boldsymbol{D}_i) \bm{x}. 
\end{equation*}
Here, $x_i \in \mathbb{R}^{m_i}$ represents the $i^{\text{th}}$ coordinate block of $\bm{x} \in \mathbb{R}^m$.  
The vector $\bm{x}_{i}$ is an extension of $x_i$ to $\mathbb{R}^m$ with all other blocks padded with zeros.  
Conversely, $\bar{\bm{x}}_i$ is the complement of $\bm{x}_i$, containing zeros in the $i^{\text{th}}$ block and ensuring that $\bm{x}_i + \bar{\bm{x}}_i = \bm{x}$.

The foundation of DCA is the surrogate convex objective obtained by linearizing the concave term in the vicinity of current estimate $\bm{x}^k$, as defined in \eqref{eqn:main-surrogate}. 
To derive a block coordinate variant of DCA, we minimize this surrogate function along specific directions determined by randomly chosen coordinate blocks. 
The algorithm operates as follows:
\begin{enumerate}[topsep=0.25em, itemsep=0.25em, leftmargin=2em]
    \item Start from a feasible initial point $ \bm{x}^0 \in \mathcal{M} $.
    \item For $ k = 1,\ldots,K $, update the estimate as follows:
    \begin{enumerate}[topsep=0.25em, itemsep=0.25em]
        \item Choose a coordinate block $ i_k $ uniformly at random from the set $ \{1,\ldots,n\} $.
        \item Keep the values fixed for all coordinates except $ i_k $, resulting in $ \smash{\bar{\bm{x}}^{k+1}_{i_k} = \bar{\bm{x}}^{k}_{i_k} }$.
        \item Update the $\smash{i_k^{\text{th}}}$ block by minimizing the surrogate objective \eqref{eqn:main-surrogate} along these coordinates. 
    \end{enumerate}
\end{enumerate}
Step 2c in this algorithm description amounts to solving the following subproblem:
\begin{equation*}
     \min_{\bm{x}\in \mathcal{M}} ~~ \hat{\phi}(\bm{x},\bm{x}^k)  ~~~ \mathrm{subj.to} ~~~ \bm{x} = \bm{x}_{i_k} + \bar{\bm{x}}^{k}_{i_k},
\end{equation*}
which is clearly equivalent to solving the following:
\begin{equation*}
    \min_{x_{i_k} \in \mathcal{M}_{i_k}} ~~ \hat{\phi}( {\boldsymbol{D}}_{i_k}^\top x_{i_k} + \bar{\bm{x}}^{k}_{i_k},\bm{x}^k). 
\end{equation*}
To simplify the presentation, we introduce the notation $\hat{\phi}_i: \mathbb{R}^{m_i} \times \mathbb{R}^m \to \mathbb{R}$ for the surrogate objective restricted to the $i^{\text{th}}$ coordinate block:
\begin{equation*}
\begin{aligned}
    & \hat{\phi}_{i}(x_i, \bm{x}^k) := \hat{\phi}( {\boldsymbol{D}}_{i}^\top x_i + \bar{\bm{x}}^{k}_{i},\bm{x}^k) \\
    & \qquad = \hat{\phi}(\bm{x}_i + \bar{\bm{x}}^{k}_{i},\bm{x}^k)\\
    & \qquad = \textcolor{blue}{f(\bm{x}_i + \bar{\bm{x}}^{k}_{i}) + g_i(x_i)} - g_i(x_i^k) + g(\bm{x}^k) - h(\bm{x}^k) \textcolor{blue}{ - \ip{\bm v_i^k}{x_i \textcolor{black}{- x_i^k}}},
\end{aligned}
\end{equation*}
 { where $\bm v_i^k \in \partial_i h(\bm{x}^k)$. }Here, only the terms in blue are relevant for solving the subproblem, as the remaining terms are constant with respect to $x_i$. 
Building on this discussion, we propose Block Coordinate DC Algorithm (\algo), as outlined in \Cref{alg:BCDC}.

\begin{algorithm*}[tb]
   \caption{Block Coordinate DC Algorithm (\algo)}
   \label{alg:BCDC} 
\begin{algorithmic}
{\setstretch{1.15}
   \STATE {\bfseries Input:} Starting point $\bm{x}^{1}\subseteq \mathcal{M}$, and total number of itreations $K$. %
   \FOR{ $k=1$ {\bfseries to} $K$}
        \STATE Select a subgradient $\bm{v}^k \in \partial h(\bm{x}^k)$ 
        \STATE Choose $i_k$ from $\{1,\ldots,n\}$ uniformly at random
        \STATE Find $ x_{i_k}^{k+1} \in \argmin_{x_{i_k} \in \mathcal{M}_{i_k}} f(\bm{x}_{i_k} + \bar{\bm{x}}_{i_k}^{k}) + g_{i_k}(x_{i_k})-\langle v^k_{i_k},x_{i_k} \rangle$
        \STATE Set $\bm{x}^{k+1} =\bm{x}_{i_k}^{k+1} + \bar{\bm{x}}_{i_k}^{k}$
   \ENDFOR
   \STATE {\bfseries Output:} $\bm{x}^{k+1}$
}
\end{algorithmic}
\end{algorithm*}

\subsection{Convergence Guarantees for BDCA}
In this section we analyze the convergence guarantees of \algo. 
Specifically, we establish an upper bound on a stationarity measure for \algo and use it to prove the convergence rate of \Cref{alg:BCDC}.  {For clarity of exposition, we first develop the analysis assuming $h$ is differentiable, which yields convergence to first-order stationary points. We then show that the same analysis extends to the non-differentiable case, where the guarantee weakens to convergence to critical points.}

\begin{definition}%
\label{def:gap}
We use the following `gap' function as a measure of closeness to a first order stationary point:
\begin{equation*}
   \gap^L_{\mathcal{M}}(\bm{y}) = \max_{\bm{x} \in \mathcal{M}}  \bigg\{ \ip{\nabla f(\bm{y})  - \nabla h(\bm{y})}{ \bm{y} - \bm{x}}  
   + g(\bm{y}) - g(\bm{x}) - \frac{L}{2} \norm{\bm{x} - \bm{y}}^2 \bigg\}.  
\end{equation*}
\end{definition}
 {To the best of our knowledge, this measure of stationarity has not been explored in DC formulations.
It is closely related to the Forward–Backward envelope (see Section~2.1 in \cite{fatkhullin2024taming} and references therein).}
The next lemma shows that this is a suitable measure for first-order stationarity. 

\begin{lemma} \label{lem:gap}
 {Let $\mathcal{M} \subseteq \mathbb{R}^m$ be a closed convex set, and let $f,g,h : \mathcal{M} \to \mathbb{R}$ be lower-semicontinuous convex functions, where $f$ is $L$-smooth and $h$ is continuously differentiable.}
Then, $\gap^L_{\mathcal{M}}(\bm{y}) \geq 0$ for all $\bm{y}\in\mathcal{M}$, and $\gap^L_{\mathcal{M}}(\bm{y}) = 0$ if and only if $\bm{y}$ is a first-order stationary point of problem~\eqref{eqn:main}. 
\end{lemma}

\begin{proof}
    The first statement is straightforward, since the  expression being maximized becomes zero when we set $\bm{x} = \bm{y}$, and $\bm{y} \in \mathcal{M}$. 
    
    For the second statement, let us recall the definition of first-order stationarity. 
    A point $\bm{y}$ is said to be a first-order stationary point of problem~\eqref{eqn:main} if the following condition holds:
    \begin{equation}\label{eqn:fo-stat}
        0 \in \nabla f(\bm{y}) + \partial g(\bm{y}) - \nabla h(\bm{y}) + \mathcal{N}_{\mathcal{M}}(\bm{y}).
    \end{equation}
    Here, $\partial g$ is the subdifferential of $g$ and $\mathcal{N}_{\mathcal{M}}$ represents the normal cone of $\mathcal{M}$. 

    If \eqref{eqn:fo-stat} holds, it implies there exists $\bm{u} \in \partial g(\bm{y})$ such that
    \begin{align}
        &  -\nabla f(\bm{y}) - \bm{u} + \nabla h(\bm{y}) \in \mathcal{N}_{\mathcal{M}}(\bm{y}) \label{eqn:stationary-cond} \\[0.5em]
        \iff ~ & \ip{\nabla f(\bm{y}) + \bm{u} - \nabla h(\bm{y})}{\bm{y} - \bm{x}} \leq 0, ~ \forall \bm{x} \in \mathcal{M} \notag\\[0.5em]
        \overset{\text{convex $g$}}{\implies} ~ & \ip{\nabla f(\bm{y}) - \nabla h(\bm{y})}{\bm{y} - \bm{x}} + g(\bm{y}) - g(\bm{x}) \leq 0,  \notag\\[0.5em]
        \implies ~ & \smash{\ip{\nabla f(\bm{y}) - \nabla h(\bm{y})}{\bm{y} - \bm{x}} + g(\bm{y}) - g(\bm{x})} - \tfrac{L}{2} \norm{\bm{x} - \bm{y}}^2 \leq 0, \notag
        \\[0.5em]
        \iff ~ & \smash{\gap^L_{\mathcal{M}}(\bm{y}) \leq 0}, \notag
    \end{align}
      {for all $\bm x\in\mathcal{M}$, where the second line follows from the definition of normal cone.} 
     Combining this with the first statement, which asserts that $\smash{\gap^L_{\mathcal{M}}(\bm{y}) \geq 0}$ for all $\bm{y} \in \mathcal{M}$, we conclude that $\smash{\gap^L_{\mathcal{M}}(\bm{y}) = 0}$ if $\bm{y}$ is a first-order stationary point. 

    Next, we show that if $\smash{\gap^L_{\mathcal{M}}(\bm{y}) = 0}$, then $\bm{y}$ is a first-order stationary point. 
    Suppose $\smash{\gap^L_{\mathcal{M}}(\bm{y}) = 0}$. Then, 
    \begin{equation*}
        \ip{\nabla f(\bm{y}) - \nabla h(\bm{y})}{ \bm{x} - \bm{y}} + g(\bm{x}) - g(\bm{y}) + \frac{L}{2} \norm{\bm{x} - \bm{y}}^2 \geq 0, 
    \end{equation*}
    for all $\bm{x} \in \mathcal{M}$. Consider $\bm{x} = \bm{y} + \alpha \bm{d}$ for an arbitrary feasible direction $\bm{d}$ (unit norm) and step-size $\alpha > 0$:
    \begin{equation*}
    \begin{aligned}
         \alpha \ip{\nabla f(\bm{y}) - \nabla h(\bm{y})}{ \bm{d}} + g(\bm{y}&+\alpha\bm{d}) - g(\bm{y}) + \frac{L}{2} \alpha^2 \geq 0, \quad \forall \alpha\bm{d}: \bm{y} + \alpha \bm{d} \in \mathcal{M}.
    \end{aligned}
    \end{equation*}
    By dividing both sides by $\alpha$ and taking the limit as ${\alpha \to 0^+}$, we obtain
    \begin{align*}
         \ip{\nabla f(\bm{y}) - \nabla h(\bm{y})}{ \bm{d}} &+ \ip{\bm{u}}{\bm{d}} \geq 0, \qquad \forall \bm{d}: \lim_{\alpha \to 0^+} \left(\bm{y} + \alpha \bm{d}\right) \in \mathcal{M}, 
    \end{align*}
    for some $\bm{u} \in \partial g(\bm{y})$. 
    Since $\mathcal{M}$ is a closed and convex set, we have $\lim_{\alpha \to 0^+} (\bm{y}+\alpha\bm{d}) \in \mathcal{M}$ for all $\bm{d} = \bm{x} - \bm{y}$ such that $\bm{x} \in \mathcal{M}$. 
    Consequently, we have
    \begin{align*}
        \ip{\nabla f(\bm{y}) - \nabla h(\bm{y})}{ \bm{x} - \bm{y}} + \ip{\bm{u}}{\bm{x} - \bm{y}} \geq 0, 
    \end{align*}
    for all $\bm{x} \in \mathcal{M}$. This is equivalent to \eqref{eqn:stationary-cond}, thus concludes the proof. 
\end{proof}

\begin{remark}\label{rem:gap_prox}
    We can relate the gap function in \Cref{def:gap} to the Moreau-Yosida envelope:
    $\gap^L_{\mathcal{M}}(\bm{y}) = 0$ if and only if $\bm{y}$ is a fixed point of the proximal mapping:
    \begin{equation*}
        \bm{y} = \prox_{\frac{1}{L} g+I_{\mathcal{M}}} \left((\bm{y}-\tfrac{1}{L}\big(\nabla f(\bm{y}) - \nabla h(\bm{y}) \big) \right) 
    \end{equation*}
    where $I_{\mathcal{M}}$ is the indicator function of the set $\mathcal{M}$. 
     {
    Moreover, $\gap^L_{\mathcal{M}}(\bm{y})$ provides a stronger measure than the proximal gradient mapping residual, satisfying
    \begin{equation*}
        \gap^L_{\mathcal{M}}(\bm{y}) \geq \frac{L}{2} \left\| \bm{y} - \prox_{\frac{1}{L} g+I_{\mathcal{M}}} \left((\bm{y}-\tfrac{1}{L}\big(\nabla f(\bm{y}) - \nabla h(\bm{y}) \big) \right) \right\|^2.
    \end{equation*}
    }
    The proof of this connection is deferred to the appendix. 
\end{remark}

We can now present the main convergence result. 

\begin{theorem}%
\label{thm:theorem-gap}
    Suppose $\bm{x}^1, \ldots, \bm{x}^K$ is a sequence generated by \algo for solving problem~\eqref{eqn:main}. 
    Then, the following bound holds:
    \begin{equation*}
        \min_{k\in\{1,\ldots,K\}} \expect{\gap^L_{\mathcal{M}}(\bm{x}^k)}
        \leq \frac{n}{K} \left(\phi(\bm{x}^1) - \phi^\star\right).
    \end{equation*}
\end{theorem}

\begin{proof}

We begin by noting that $\forall x_{i_k} \in \mathcal{M}_{i_k}$, we have
\begin{align}
     \phi(\bm{x}^{k+1}) \leq &  {f(\bm{x}^{k+1}) + g(\bm{x}^{k+1}) - h(\bm{x}^{k}) - \ip{\nabla h(\bm{x}^{k})}{\bm{x}^{k+1} - \bm{x}^{k}}}\nonumber\\
      {=} & {f(\bm{x}^{k+1}) + g(\bm{x}^{k+1}) - h(\bm{x}^{k}) - \ip{\nabla_{i_k} h(\bm{x}^{k})}{{x}_{i_k}^{k+1} - {x}_{i_k}^{k}}}\nonumber\\
    =& \hat{\phi}(\bm{x}_{i_k}^{k+1} + \bar{\bm{x}}^{k}_{i},\bm{x}^k)
        = \hat{\phi}_{i_k}(x_{i_k}^{k+1}, \bm{x}^k)  \leq \hat{\phi}_{i_k}(x_{i_k}, \bm{x}^k) \nonumber\\
         = &f(\bar{\bm{x}}_{i_k}^k + \bm{x}_{i_k}) + g(\bm{x}^k) + g_{i_k}(x_{i_k}) - g_{i_k}(x_{i_k}^k) - h(\bm{x}^k) \nonumber\\
        & - \ip{\nabla_{i_k} h(\bm{x}^k)}{x_{i_k} - x_{i_k}^k}. \nonumber
\end{align}
Here, the first inequality holds due to the convexity of $h$, while the second inequality follows from the definition of $x_{i_k}^{k+1}$ in the update rule of \algo. 
Adding $f(\bm{x}^k)$ to both sides and rearranging, we get
\begin{equation}
\begin{aligned}
    f(\bm{x}^k) - f(\bar{\bm{x}}_{i_k}^k + \bm{x}_{i_k}) + g_{i_k}(x_{i_k}^k) - g_{i_k}(x_{i_k}) &+ \ip{\nabla_{i_k} h(\bm{x}^k)}{x_{i_k} - x_{i_k}^k}\nonumber\\
    &\leq \phi(\bm{x}^k) - \phi(\bm{x}^{k+1}).
\end{aligned}
\end{equation}
Since the function $f$ is $L$-smooth, we have
\begin{equation*}
    f(\bar{\bm{x}}_{i_k}^k + \bm{x}_{i_k})
     \leq f(\bm{x}^k) + \ip{\nabla_{i_k} f(\bm{x}^k)}{x_{i_k} - x_{i_k}^k} + \frac{L}{2} \norm{x_{i_k} - x_{i_k}^k}^2. 
\end{equation*}
Combining the last two inequalities, we obtain
\begin{align}\label{eqn:proof-main-step}
     \ip{\nabla_{i_k} f(\bm{x}^k) - \nabla_{i_k} h(\bm{x}^k)}{ x_{i_k}^k - x_{i_k}} + g_{i_k}(x_{i_k}^k) - & g_{i_k}(x_{i_k}) - \frac{L}{2} \norm{x_{i_k} - x_{i_k}^k}^2
    \nonumber\\
    &\leq \phi(\bm{x}^k) - \phi(\bm{x}^{k+1}).
\end{align}

Let us denote by $\mathbb{E}_k$ the conditional expectation with respect to the random selection of $i_k$, given all the random choices in the previous iterations. Then,   
\begin{align}
    \mathbb{E}_{k}\bigg[\ip{&\nabla_{i_k} f(\bm{x}^k) - \nabla_{i_k} h(\bm{x}^k)}{ x_{i_k}^k - x_{i_k}} + g_{i_k}(x_{i_k}^k) - g_{i_k}(x_{i_k}) - \frac{L}{2} \norm{x_{i_k} - x_{i_k}^k}^2\bigg] \nonumber\\
     = \frac{1}{n} &\sum_{i=1}^n \bigg( \ip{\nabla_i f(\bm{x}^k) - \nabla_i h(\bm{x}^k)}{ x_i^k - x_i}  + g_i(x_i^k) - g_i(x_i) - \frac{L}{2} \norm{x_i - x_i^k}^2\bigg) \nonumber\\
     = \frac{1}{n} &\bigg(  \ip{\nabla f(\bm{x}^k) - \nabla h(\bm{x}^k)}{ \bm{x}^k - \bm{x}} + g(\bm{x}^k) - g(\bm{x}) - \frac{L}{2} \norm{\bm{x} - \bm{x}^k}^2 \bigg). \nonumber
\end{align}
Combining this with \eqref{eqn:proof-main-step} gives the following bound, which holds for all $\bm{x} \in \mathcal{M}$:
\begin{align}
    \ip{\nabla f(\bm{x}^k) - \nabla h(\bm{x}^k)}{ \bm{x}^k - \bm{x}} + g(\bm{x}^k) & - g(\bm{x}) - \frac{L}{2} \norm{\bm{x} - \bm{x}^k}^2  \nonumber\\
    & \leq n \phi(\bm{x}^k) - n \expectk{\phi(\bm{x}^{k+1})}.
\end{align}
Now, we maximize this inequality over $\bm{x} \in \mathcal{M}$ to get
\begin{equation*}
    \gap^L_{\mathcal{M}}(\bm{x}^k)
    \leq n \phi(\bm{x}^k) - n \expectk{\phi(\bm{x}^{k+1})}.
\end{equation*}
Then, we take the expectation of both sides over the random choices in all iterations:
\begin{equation}\label{eqn:proof_mai_thm_10}
    \expect{\gap^L_{\mathcal{M}}(\bm{x}^k)}
    \leq n \expect{\phi(\bm{x}^k)} - n \expect{\phi(\bm{x}^{k+1})}.
\end{equation}
Finally, we average this inequality over $k=1,\ldots,K$:
\begin{align}
    \frac{1}{K} \sum_{k=1}^K \expect{\gap^L_{\mathcal{M}}(\bm{x}^k)}
    \leq \frac{n}{K} \left(\phi(\bm{x}^1) - \expect{\phi(\bm{x}^{K+1})})\right) 
    \leq \frac{n}{K} \left(\phi(\bm{x}^1) - \phi^\star\right).
\end{align}
We complete the proof by noting that the minimum of $\gap^L_{\mathcal{M}}(\bm{x}^k)$ over $k=1,\ldots,K$ is smaller than or equal to the average gap.
\end{proof}

 {
\begin{remark}
    The smoothness assumption of $f$ is not required by \Cref{alg:BCDC} itself, but only for the analysis. 
    In particular, smoothness allows us to bound each block update of $f$ by a quadratic model,
    \begin{equation*}
        \langle \nabla_{i} f(\bm{x}^k), x_{i} - x_{i}^k \rangle + \tfrac{L}{2}\|x_{i} -     x_{i}^k\|^2
    \;\leq\; f(\bm{x}^k) - f(\bar{\bm{x}}_{i}^k + \bm{x}_{i}),
    \end{equation*}
    which effectively gives a block-separable structure for the analysis. 
    This step is essential in our proof but does not play any role in the algorithm itself.
\end{remark}
}

\begin{remark}
    The rate in \Cref{thm:theorem-gap} is consistent with existing results for CD methods in convex optimization. 
    For instance, Theorem~1 in \cite{wright2015coordinate} shows that $\mathbb{E} [f(\bm{x}^k)] - f^\star \leq \frac{2nL\norm{\bm{x}^1-\bm{x}_\star}^2}{k}$ for the CD method in unconstrained smooth convex minimization. 
    In terms of first-order stationarity condition, this implies 
    \begin{equation*}
                \expect{\norm{\nabla f(\bm{x}^k)}^2} 
                \leq 2 L \expect{f(\bm{x}^k) - f^\star} 
                \leq \frac{4nL^2\norm{\bm{x}^1-\bm{x}_\star}^2}{k}.
    \end{equation*}
    \Cref{thm:theorem-gap} yields a comparable result, since in this setting $\phi = f$ satisfies $f(\bm{x}^1) - f(\bm{x}^\star) \leq \frac{L}{2}\norm{\bm{x}^1-\bm{x}^\star}^2$, and the gap takes the form of $\gap^L(\bm{x}^k) = \tfrac{1}{2L} \norm{\nabla f(\bm{x}^k)}^2$. 
    Thus, we obtain 
    \begin{equation*}
        \min_{k \in \{1,\ldots,K\}} \expect{\norm{\nabla f(\bm{x}^k)}^2}
        \leq \frac{nL^2\norm{\bm{x}^1-\bm{x}_\star}^2}{K}.
    \end{equation*}
\end{remark}

Identifying the gap function in \Cref{def:gap} as a measure of first-order stationarity is an  {important technical component} of our analysis. 
Notably,  {many} standard  {stationarity} measures are not suitable for nonsmooth nonconvex optimization problems, motivating recent efforts to explore and reassess alternative definitions  \cite{zhang2020complexity,kornowski2021oracle}. 
For example, the distance measure
\begin{equation*}
    \mathrm{dist}(0,\nabla f(\bm{y}) + \partial g(\bm{y}) - \nabla h(\bm{y}) + \mathcal{N}_{\mathcal{M}}(\bm{y})),
\end{equation*}
or the `Frank-Wolfe gap' using a subgradient $\bm{u} \in \partial g(\bm{y})$
\begin{equation*}
    \max_{\bm{x} \in \mathcal{M}} ~  \ip{\nabla f(\bm{y}) + \bm{u} - \nabla h(\bm{y})}{ \bm{y} - \bm{x}}   
\end{equation*}
are intractable, meaning that no algorithm can guarantee finding a point satisfying $\epsilon$-closeness in these notions within a finite number of iterations. 

Our results  {suggest} that the challenge of identifying appropriate stationarity measures stems from the nonsmooth convex term, while the nonsmooth concave term remains benign. 
Although it may initially seem surprising, this can be explained by the fact that the concave term does not invalidate the smoothness-based quadratic upper bound, since
\begin{align*}
    f(\bm{y}) & \leq f(\bm{x}) + \ip{\nabla f(\bm{x})}{\bm{y} - \bm{x}} + \frac{L}{2}\norm{\bm{x}-\bm{y}}^2 \\
    - h(\bm{y}) & \leq - h(\bm{x}) - \ip{\nabla h(\bm{x})}{\bm{y} - \bm{x}}
\end{align*}
for all $\bm{x}$ and $\bm{y}$. 
Summing these two inequalities, we see that $f(\bm{y})-h(\bm{y})$ behaves similarly to a smooth function in the analysis, as the quadratic upper bound is  {often} the only property required from the smoothness assumption.

 {

\subsubsection{Guarantees under Non-differentiable Concave Terms}

When $h$ is non-differentiable, we replace $\nabla h(\bm{y})$ in the definition of the gap by a subgradient $\bm{v} \in \partial h(\bm{y})$, and define
\begin{equation*}
    \gap^L_{\mathcal{M}}(\bm{y}) = \max_{\bm{x} \in \mathcal{M}} \min_{\bm{v} \in \partial h(\bm{y})}  \big\{ \ip{\nabla f(\bm{y})  - \bm{v}}{ \bm{y} - \bm{x}} + g(\bm{y}) - g(\bm{x}) - \frac{L}{2} \norm{\bm{x} - \bm{y}}^2 \big\}.  
\end{equation*}
This choice ensures the gap remains well-defined, as it is sufficient that the condition holds for one subgradient of $h$ at $\bm y$. 
When $h$ is differentiable, this definition reduces to \Cref{def:gap}.

The result below extends \Cref{lem:gap} to the setting where $h$ is non-differentiable; the proof follows similar arguments and is deferred to the appendices

\begin{lemma} \label{lem:gap-nonsmooth}
Let $\mathcal{M} \subseteq \mathbb{R}^m$ be a closed convex set, and let $f,g,h : \mathcal{M} \to \mathbb{R}$ be lower-semicontinuous convex functions, where $f$ is $L$-smooth. 
Then $\gap^L_{\mathcal{M}}(\bm{y}) \geq 0$ for all $\bm{y}\in\mathcal{M}$, and $\gap^L_{\mathcal{M}}(\bm{y}) = 0$ if and only if $\bm{y}$ is a critical point of problem~\eqref{eqn:main}, defined by the condition
\begin{equation}\label{eqn:criticality}
    \big(\nabla f(\bm{y}) + \partial g(\bm{y}) + \mathcal{N}_{\mathcal{M}}(\bm{y})\big) \cap \partial h(\bm{y}) \neq \varnothing.
\end{equation}
\end{lemma}

\Cref{thm:theorem-gap} continues to hold in the non-differentiable setting. 
The proof follows identically, except that $\nabla_{i_k} h(\bm{x}^k)$ replaced by the $i_k$-th component (or block) of a subgradient $\bm{v}^k \in \partial h(\bm{x}^k)$, denoted by $v^k_{i_k}$. 
Any subgradient can be used, since the gap is defined via a minimization over the subdifferential set. 
The proof proceeds without further modification.
}

 {

\subsubsection{Faster Rates Under a Generalized Polyak-Łojasiewicz Condition}

In the same spirit as \citep{Karimi2016LinearCO}, it is possible to prove a linear convergence rate for \algo, conditioned on a generalized Polyak–Łojasiewicz (PL) inequality defined in terms of our gap function. 
\begin{theorem}
    Assume that $\phi$ satisfies the generalized PL inequality, i.e., there exists $\mu \in (0,L]$ such that 
    \begin{equation}\label{eqn:genPL}
    \phi(\bm x) - \phi(\bm x^*) \le \frac{L}{\mu} \gap_{\mathcal{M}}^L(\bm x)
        \quad \text{for all } \bm{x} \in \mathcal{M}.
    \end{equation} 
    Suppose $\bm{x}^1, \ldots, \bm{x}^K$ is a sequence generated by \algo for solving problem~\eqref{eqn:main}. 
    Then, the following bound holds:
    \begin{equation*}
        \expect{\phi(\bm{x}^{k+1})} - \phi(\bm x^*)
        \leq \left(1 - \frac{\mu}{nL}\right)^{k+1}\left(\expect{\phi(\bm x^0)} - \phi(\bm x^*)\right).
    \end{equation*} 
\end{theorem}
\begin{proof}
    The proof starts by following the same steps as that of \Cref{thm:theorem-gap}. 
    Then, rearranging \eqref{eqn:proof_mai_thm_10}, subtracting $\phi(\bm x^*)$, and applying \eqref{eqn:genPL}, we obtain
    \begin{align*}
        \expect{\phi(\bm{x}^{k+1})} - \phi(\bm x^*) 
        \leq & \expect{\phi(\bm{x}^{k})} - \phi(\bm x^*) 
        -\frac{1}{n}\expect{\gap_{\mathcal{M}}^L(\bm x^k)} \\
         \leq & \left(1 - \frac{\mu}{nL}\right)\left(\expect{\phi(\bm x^k)} - \phi(\bm x^*)\right).
    \end{align*}   
    Iterating this inequality concludes the proof.
\end{proof}

}

\section{Block EM Algorithm}\label{sec:EM}

The Expectation-Maximization (EM) algorithm is a cornerstone in statistical analysis for finding maximum likelihood estimates in the presence of incomplete data; widely used in applications such as mixture models, latent variable models, data imputation, and clustering  \cite{dempster1977maximum,mclachlan2007algorithm}. 
EM shares strong connections with DCA \cite{yuille2003concave,lethi30years}. 
In this section, we leverage these connections to develop a block EM algorithm based on our proposed \algo. 

Consider negative log-likelihood minimization problem:
\begin{equation} \label{eqn:ML}
    \min_{\bm{\theta}} ~ \mathcal{L}(\bm{\theta}) := - \sum_{\bm{x} \in \mathcal{X}}\log P(\bm{x}|{\bm{\theta}}). %
\end{equation}
Note the change in notation: here, $\bm{x}$ denotes data points, while the decision variable $\bm{\theta} \in \Theta \subseteq \R^m$ represents the parameters of the underlying distribution. 
The dataset $\mathcal{X}$ has support size $N$, and $\mathcal{L}$ is the negative log-likelihood function. 
The distribution $P(\bm{x}|{\bm{\theta}})$ is defined through a hidden variable $\bm{y}$:
\begin{equation*}
    P(\bm{x}|{\bm{\theta}}) = \sum_{\bm{y}} P(\bm{x},\bm{y}|{\bm{\theta}}),
\end{equation*}
where $P(\bm{x},\bm{y}|{\bm{\theta}})$ denotes the complete-data likelihood. 
The presence of hidden variables complicate \eqref{eqn:ML}, making standard optimization techniques inapplicable. 

The EM algorithm, for solving this problem, iteratively alternates between two steps -- Expectation (E-step) and Maximization (M-step) -- to improve the likelihood of the model parameters until convergence:
\begin{align}
    Q(\bm \theta;\bm{\theta}^k) & ~=~ \mathbb{E}_{P(\bm{y}|\bm{x},{\bm{\theta}^k} {)}}\left[ \log P(\bm{x},\bm{y}|{\bm{\theta}}) \right] \tag{E-step}\\
    \bm{\theta}^{k+1} & ~=~ \underset{\bm{\theta}\in\Theta}{\arg\max} ~ Q(\bm{\theta};\bm{\theta}^k) \tag{M-step}
\end{align}
Here, $P(\bm{y}|\bm{x},{\bm{\theta}^k})$ is the posterior distribution of the latent variables given the observed data and the current parameter estimate $\bm{\theta}^k$. 
In the E-step, the expected value of the complete-data log-likelihood is computed with respect to the conditional distribution of the latent variables given the observed data and the current parameter estimates. 
In the M-step, this expected log-likelihood is maximized to update the parameters. 

For the broad class of natural exponential family distributions, EM is a specific instance of DCA. 

\begin{definition}\label{def:expo-family}
    The natural exponential family of distributions with a parameter $\bm{\theta}\in\Theta$ is characterized by the following probability mass (or density) function:
    \begin{equation} \label{eqn:nat-expo-dist}
        P(\bm{x},\bm{y}|{\bm{\theta}}) = \frac{\varphi(\bm{x},\bm{y}) e^{\ip{\bm{\theta}}{T(\bm{x},\bm{y})}}}{\sum_{\bm{x},\bm{y}} \varphi(\bm{x},\bm{y}) e^{\ip{\bm{\theta}}{T(\bm{x},\bm{y})}}},
    \end{equation}
    where $\varphi$ is a positive function (base measure) and $T$ denotes the sufficient statistic. %
\end{definition}

It is easy to see that problem \eqref{eqn:ML} with a natural exponential family distribution \eqref{eqn:nat-expo-dist} can be written as a DC program, since the log-likelihood function can be decomposed as $\mathcal{L} \!=\! f \!-\! h$, where the terms are given by
\begin{equation} \label{eqn:functions}
\begin{aligned}
    f(\bm \theta) & = N\log\Big(\sum_{\bm{x},\bm{y}} \varphi(\bm{x},\bm{y}) e^{\ip{\bm{\theta}}{T(\bm{x},\bm{y})}} \Big),\\ 
    h(\bm \theta) & = \sum_{x\in\mathcal{X}}\log\Big(\sum_{\bm y} \varphi(\bm{x},\bm{y}) e^{\ip{\bm{\theta}}{T(\bm{x},\bm{y})}} \Big).
\end{aligned}
\end{equation}
Both terms are convex and smooth, owing to their log-sum-exp structure. 

Applying DCA to this problem leads to the following update rule: 
\begin{align}\label{eqn:em_deriv}
        \bm{\theta}^{k+1} \in \argmin_{ \bm{\theta}} ~~ &N\log\left(\sum_{\bm{x}, \bm{y}} \varphi(\bm{x},\bm{y})e^{\ip{\bm{\theta}}{T(\bm{x}, \bm{y})}}\right)-  \sum_{x\in\mathcal{X}} \frac{\sum_{\bm{y}} \varphi(\bm{x},\bm{y})\ip{\bm{\theta}}{T(\bm{x}, \bm{y})}e^{\ip{\bm{\theta}^k}{T(\bm{x}, \bm{y})}}}{\sum_{\bm{y}} \varphi(\bm{x},\bm{y})e^{\ip{\bm{\theta}^k}{T(\bm{x}, \bm{y})}}}. 
\end{align}
 {Note that the dataset size $N$ can be expressed as: 
\begin{equation*}
    N = \sum_{x\in\mathcal{X},\bm{y}}\frac{\varphi(\bm{x},\bm{y})e^{\ip{\bm{\theta}^k}{T(\bm{x}, \bm{y})}}}{\sum_{\bm{y}}\varphi(\bm{x},\bm{y})e^{\ip{\bm{\theta}^k}{T(\bm{x}, \bm{y})}}}.
\end{equation*}
Also, we can write $\ip{\bm{\theta}}{T(\bm{x}, \bm{y})}  = \log(e^{\ip{\bm{\theta}}{T(\bm{x}, \bm{y})}})$, and add $-N\log(\varphi(\bm{x},\bm{y}))$ to~\eqref{eqn:em_deriv}, as it is constant with respect to $\bm \theta$. 
This gives:
\begin{align}
        \bm{\theta}^{k+1} \in \argmin_{ \bm{\theta}} ~~ &\sum_{x\in\mathcal{X},\bm{y}}\frac{\varphi(\bm{x},\bm{y})e^{\ip{\bm{\theta}^k}{T(\bm{x}, \bm{y})}}}{\sum_{\bm{y}}\varphi(\bm{x},\bm{y})e^{\ip{\bm{\theta}^k}{T(\bm{x}, \bm{y})}}}\log\left(\sum_{\bm{x}, \bm{y}} \varphi(\bm{x},\bm{y})e^{\ip{\bm{\theta}}{T(\bm{x}, \bm{y})}}\right)\nonumber\\ 
        &-  \sum_{x\in\mathcal{X},\bm{y}} \frac{ \varphi(\bm{x},\bm{y})e^{\ip{\bm{\theta}^k}{T(\bm{x}, \bm{y})}}}{\sum_{\bm{y}} \varphi(\bm{x},\bm{y})e^{\ip{\bm{\theta}^k}{T(\bm{x}, \bm{y})}}}\log\left(\varphi(\bm{x},\bm{y})e^{\ip{\bm{\theta}}{T(\bm{x}, \bm{y})}}\right),\nonumber
\end{align}
or equivalently,
}
\begin{align}\label{eqn:dc-em}
    \bm{\theta}^{k+1} \in \argmin_{\bm{\theta}} \sum_{x\in\mathcal{X}}\sum_{\bm{y}}\frac{\varphi(\bm{x},\bm{y})e^{\ip{\bm{\theta}^k}{T(\bm{x}, \bm{y})}}}{\sum_{\bm{y}}\varphi(\bm{x},\bm{y})e^{\ip{\bm{\theta} {^k}}{T(\bm{x}, \bm{y})}}}
    \Bigg[\log\left(\frac{\sum_{\bm{x}, \bm{y}} \varphi(\bm{x},\bm{y})e^{\ip{\bm{\theta}}{T(\bm{x}, \bm{y})}}}{\varphi(\bm{x},\bm{y})e^{\ip{\bm{\theta}}{T(\bm{x}, \bm{y})}}}\right)
     \Bigg],
\end{align}
 {which} is exactly the EM algorithm for the exponential family. 
 {This can be seen through \Cref{def:expo-family} and noting that
\begin{align*}
    P(\bm{x},\bm{y}|{\bm{\theta}}) = \frac{\varphi(\bm{x},\bm{y}) e^{\ip{\bm{\theta}}{T(\bm{x},\bm{y})}}}{\sum_{\bm{x},\bm{y}} \varphi(\bm{x},\bm{y}) e^{\ip{\bm{\theta}}{T(\bm{x},\bm{y})}}}, 
    \quad 
    P(\bm{y}|\bm{x},{\bm{\theta}^k}) = \frac{\varphi(\bm{x},\bm{y}) e^{\ip{\bm{\theta}^k}{T(\bm{x},\bm{y})}}}{\sum_{\bm{y}} \varphi(\bm{x},\bm{y}) e^{\ip{\bm{\theta}^k}{T(\bm{x},\bm{y})}}}.
\end{align*}
} 

We then design the \algoem method by addressing the same problem using \algo. 
The resulting algorithm is presented in \Cref{alg:BCEM}. 
Notably, we can transfer the guarantees established in \Cref{thm:theorem-gap} to provide non-asymptotic convergence results for \algoem.

\begin{algorithm*}[tb]
   \caption{Block EM Algorithm}
   \label{alg:BCEM} 
\begin{algorithmic}
   \STATE {\bfseries Input:}  total iterations $K$, initialize the probability distribution $P$ using vector $\bm{\theta}^{(0)}$ ($P$ belongs to the exponential family), and number of blocks $n$
   \FOR{ $k=1$ {\bfseries to} $K$}
        \STATE Randomly choose $i_k$ in $[1,...,n]$ with uniform distribution
        \STATE Update $ \hat{P}^{k+1}(\bm{y}) =\frac{P(\bm{x},\bm{y}|\bm{\theta}^k)}{\sum_{\bm{y}} P(\bm{x},\bm{y}|\bm{\theta}^k)}$
        \STATE Update ${\theta}_{i_k}^{k+1} = \underset{\theta_{i_k}}{\mathrm{arg\,min}} -\sum_{\bm{x}\in\mathcal{X},\bm{y}}\hat{P}^{k+1}(\bm{y})\log P(\bm{x},\bm{y}|\theta^k_{1},\ldots,\theta_{i_k},\ldots,\theta^k_{n}) $
        \STATE Set $\bm \theta^{k+1} = \bar{\bm \theta}_{i_k}^{k}+ \bm \theta_{i_k}^{k+1}$
   \ENDFOR
   \STATE {\bfseries Output: $\bm{\theta}^{K+1}$}
\end{algorithmic}
\end{algorithm*}

\begin{corollary} \label{cor:block-em}
    Suppose there exists a sequence ${\bm{\theta}^k}, k =1\ldots,K$ generated by Block EM algorithm. 
    Then, 
    \begin{align*}
       \min_{k\in\{1,\ldots,K\}} \expect{\norm{\nabla \mathcal{L}(\bm{\theta}^k)}^2}
       \leq \frac{2nL}{K} \big(\mathcal{L}(\bm{\theta}^1)-\mathcal{L}(\bm{\theta}^\star)\big).
    \end{align*}
\end{corollary}

\begin{proof}
    In this setting, the gap function reduces to
    \begin{align*}
        &\max_{\bm{\theta}} ~ \ip{\underbrace{\nabla f(\bm{\theta^k})-\nabla h(\bm{\theta}^k)}_{\nabla \mathcal{L}(\bm{\theta}^k)}}{\bm{\theta}^k - \bm{\theta}} - \frac{L}{2}\norm{\bm{\theta}^k - \bm{\theta}}^2.
    \end{align*}
    The solution to this max problem is $\bm{\theta} = \bm{\theta}^k - \frac{1}{L}\nabla \mathcal{L}(\bm{\theta})$, leading to $\gap^L(\bm{\theta}^k) = \frac{1}{2L} \norm{\nabla \mathcal{L}(\bm{\theta}^k)}^2$. 
    Substituting this into \Cref{thm:theorem-gap} yields the desired result. 
\end{proof}

\begin{remark}
    The results in \Cref{cor:block-em} are consistent with the $\mathcal{O}(1/k)$ convergence rate guarantee for the standard EM algorithm, as provided by  \cite{kumar2017convergence}. 
\end{remark}

\begin{remark}
    The standard EM algorithm achieves an $\mathcal{O}(1/k)$ convergence rate in Kullback–Leibler divergence, as shown in \cite{Kunstner2020HomeomorphicInvarianceOE}.
    It is worth noting that the Lipschitz constant and the Euclidean distance generating function are not utilized in \algo; they appear in the analysis only through the definition of the gap function. 
    Consequently, we believe our analysis can be extended to the general setting of Bregman divergence. 
\end{remark}

\section{Numerical Experiments}

In this section, we demonstrate the empirical performance of \algo through numerical experiments in logistic regression with nonconvex sparsity regularization and nonconvex quadratic programming with negative $\ell_1$ regularization. 
Additionally, we examine the performance of \algoem on one-bit one-bit MIMO channel signal recovery problem in \ref{app:Mimo}. 
All experiments were conducted using MATLAB R2023b on a Intel(R) Xeon(R) Silver 4210R CPU @ 2.40GHz. 

\subsection{Logistic Regression with Nonconvex Sparsity Regularization}

The first experiment replicates the logistic regression with nonconvex sparsity regularization in \cite{deng2020efficiency}:
\begin{align}\label{eqn:classification-reg-lK}
    \min_{\bm x\in \mathbb{R}^m} ~~ \frac{1}{N}\sum_{i=1}^N \log(1+e^{-b_i\langle \bm a_i,\bm x\rangle}) + \lambda(\|\bm x\|_1 - \|\bm x\|_Q)
\end{align}
where $\|\bm x\|_Q$ denotes the largest $Q$-norm {, defined as the sum of the absolute values of the $Q$ largest components of $\bm{x}$,} and $(\bm a_i,b_i)$ represent the datapoints. 

Logistic loss is convex and smooth, and we denote by $L$ its smoothness constant. 
We compare the performance of DCA and BDCA, using the following decomposition of the objective function:
\begin{align*}
    f(\bm x)= \frac{L}{2}\|\bm x\|^2, ~~ &g(\bm x) = \lambda \|\bm x\|_1, ~~ \\
    h(\bm x) = \frac{L}{2}\|\bm x\|^2 - \frac{1}{N}\sum_{i=1}^N &\log(1+e^{-b_i\langle \bm a_i,\bm x\rangle}) + \lambda \|\bm x\|_Q. 
\end{align*}
which results in closed form DCA subproblems of the proximal (coordinate) descent type updates. 

We conduct experiments with the \textsc{Rcv1} ($N=20242,m=47236,s=0.5\%$) and \textsc{Real-Sim} ($N=702309,m=20958,s=0.25\%$) datasets from the LIBSVM library \cite{10.1145/1961189.1961199}. 
We set $\lambda = 0.1/m$ and $\lambda=10/m$, and use blocks of size $1000$.  

The results are shown in \Cref{fig:logistic_rho1,fig:logistic_rho10}. 
In this example, \algo and DCA perform similarly in terms of the number of full passes (iteration $\times$ block-size$/m$) for both the objective and the gap function (as defined in \Cref{def:gap}). 
However, \algo avoids solving full-scale subproblems, making it suitable for devices with limited computational power. 
It also offers greater potential for parallelization through asynchronous updates, which could be a valuable direction for future research.

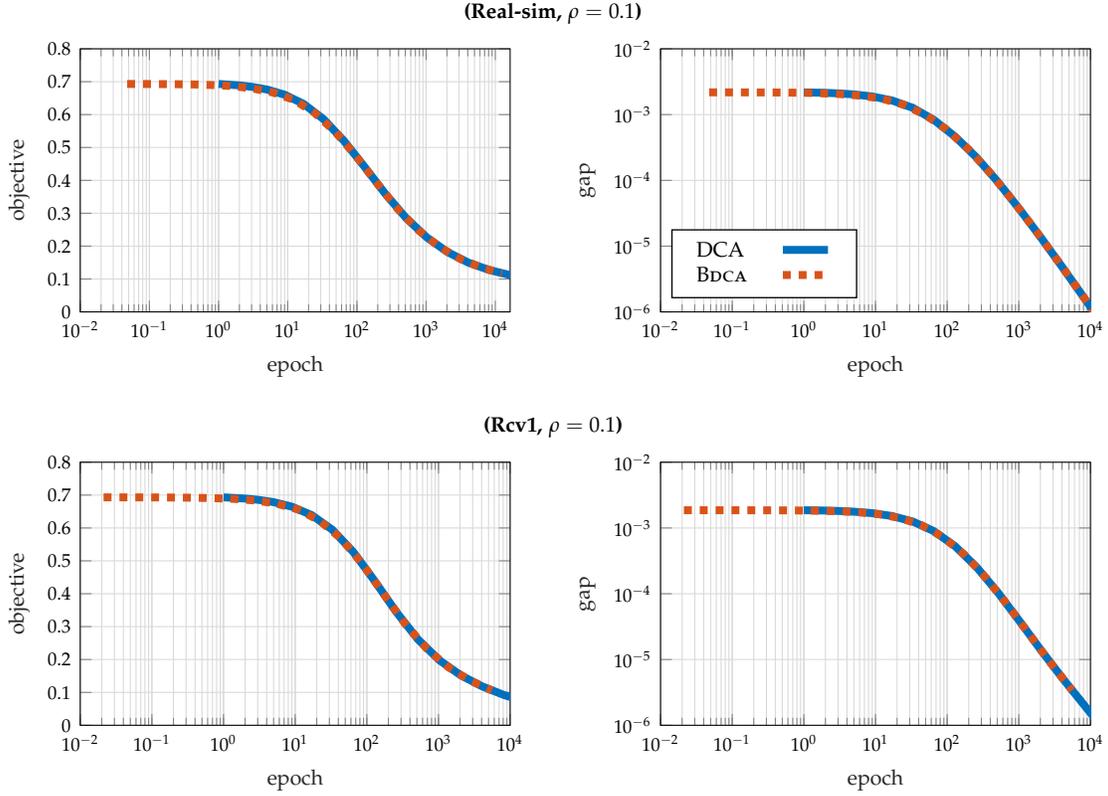
\begin{figure*}[!t]
    \centering
    \begin{tikzpicture}
    
    \begin{groupplot}[group style={group size=2 by 2 , horizontal sep = 2cm , vertical sep = 2.0cm},        
        width = 0.36\textwidth, 
        height = 3.5cm,            
        name = {figure_BDCA}
        grid = both, 
        grid style = {gray!30},        
        tick label style = {font=\scriptsize}, 
        name = figure1,
        ]

    \nextgroupplot[
scale only axis,
xmode=log,
xmin=0.01,
xmax=16385,
xminorticks=true,
ymode=linear,
ymin=0.0,
ymax=0.8,
ytick = {0,0.1,0.2,0.3,0.4,0.5,0.6,0.7,0.8},
yminorticks=true,
ylabel style={font=\footnotesize\color{white!15!black}},
xlabel style={font=\footnotesize\color{white!15!black}},
ylabel={objective},
xlabel={epoch},
axis background/.style={fill=white},
title style={font=\footnotesize\bfseries},
title={(\textsc{Real-sim}, $\rho = 0.1$)},
title style={at={(axis description cs:1.1,1)},anchor=south},
xmajorgrids,
xminorgrids,
ymajorgrids,
yminorgrids,
]
\addplot [color=mycolor1, line width=3.0pt, forget plot]
  table[row sep=crcr]{%
1	0.693147180560083\\
2	0.688806212607153\\
3	0.684541565774823\\
5	0.676244164783408\\
9	0.66052877356271\\
17	0.632255915718192\\
33	0.585881228934463\\
65	0.519971496783707\\
129	0.440847959756068\\
257	0.35914160674784\\
513	0.285567880995717\\
1025	0.226307028658688\\
2049	0.181717280425101\\
4097	0.149644483426731\\
8193	0.127302533837499\\
16385	0.111943368730555\\
};\label{fig:obj_data1_rho01_DCA}
\addplot [color=mycolor3,  dashed ,line width=3.0pt, forget plot]
  table[row sep=crcr]{%
0.0477122000095424	0.693147180560083\\
0.0954244000190849	0.693005514695953\\
0.143136600028627	0.692723106211841\\
0.238561000047712	0.692395143318214\\
0.429409800085882	0.691766760989692\\
0.811107400162221	0.689828749211011\\
1.5745026003149	0.686821879646128\\
3.10129300062026	0.679400942095649\\
6.15487380123097	0.666992660850313\\
12.2620354024524	0.644820240110093\\
24.4763586048953	0.605161322946766\\
48.905005009781	0.546856227417871\\
97.7622978195525	0.473088370676823\\
195.476883439095	0.390853000870039\\
390.906054678181	0.312347627834553\\
781.764397156353	0.24754205155639\\
1563.4810821127	0.197425041333102\\
3126.91445202538	0.160833137419523\\
6253.78119185076	0.135036779728811\\
12507.5146715015	0.117236366040403\\
};\label{fig:obj_data1_rho01_BDCA}

\nextgroupplot[
scale only axis,
xmode=log,
xmin=0.01,
xmax=10000,
xminorticks=true,
ymode=log,
ymin=1e-06,
ymax=1e-02,
ytick = {1e-2,1e-3,1e-4,1e-5,1e-6},
yminorticks=true,
ylabel style={font=\footnotesize\color{white!15!black}},
xlabel style={font=\footnotesize\color{white!15!black}},
ylabel={gap},
xlabel={epoch},
axis background/.style={fill=white},
title style={font=\footnotesize\bfseries},
xmajorgrids,
xminorgrids,
ymajorgrids,
yminorgrids,
]

\addplot [color=mycolor1, line width=3.0pt, forget plot]
  table[row sep=crcr]{%
1	0.00217958908155269\\
2	0.00214216935345935\\
3	0.00210300660781898\\
5	0.00202731125398395\\
9	0.00188603446815272\\
17	0.00164046647337939\\
33	0.0012697509215325\\
65	0.00083313153204813\\
129	0.000462885748443134\\
257	0.000222997911316463\\
513	9.41526566207425e-05\\
1025	3.63836581514416e-05\\
2049	1.33419127306269e-05\\
4097	4.70591785785392e-06\\
8193	1.62300296509055e-06\\
16385	5.54287676658968e-07\\
};\label{fig:gap_data1_rho01_DCA}
\addplot [color=mycolor3,  dashed ,line width=3.0pt, forget plot]
  table[row sep=crcr]{%
0.0477122000095424	0.00217917596921873\\
0.0954244000190849	0.00217855534937874\\
0.143136600028627	0.00217614978445663\\
0.238561000047712	0.00217330642767161\\
0.429409800085882	0.00216783827778125\\
0.811107400162221	0.00215021643098079\\
1.5745026003149	0.00212195105946635\\
3.10129300062026	0.00205470064600956\\
6.15487380123097	0.00194072195492196\\
12.2620354024524	0.00174679473130567\\
24.4763586048953	0.00141716055617524\\
48.905005009781	0.000996914033459384\\
97.7622978195525	0.000596215846978841\\
195.476883439095	0.000302741389056948\\
390.906054678181	0.000132807771613805\\
781.764397156353	5.30687840806115e-05\\
1563.4810821127	1.982191043173e-05\\
3126.91445202538	7.09390512039063e-06\\
6253.78119185076	2.46154989834893e-06\\
12507.5146715015	8.45140893038108e-07\\
};\label{fig:gap_data1_rho01_BDCA}

\nextgroupplot[
scale only axis,
xmode=log,
xmin=0.01,
xmax=10000,
xminorticks=true,
ymode=linear,
ymin=0.0,
ymax=0.8,
ytick = {0,0.1,0.2,0.3,0.4,0.5,0.6,0.7,0.8},
yminorticks=true,
ylabel style={font=\footnotesize\color{white!15!black}},
xlabel style={font=\footnotesize\color{white!15!black}},
ylabel={objective},
xlabel={epoch},
axis background/.style={fill=white},
title style={font=\footnotesize\bfseries},
title={(\textsc{Rcv1}, $\rho = 0.1$)},
title style={at={(axis description cs:1.1,1)},anchor=south},
xmajorgrids,
xminorgrids,
ymajorgrids,
yminorgrids,
]
\addplot [color=mycolor1, line width=3.0pt, forget plot]
  table[row sep=crcr]{%
1	0.693147180559993\\
2	0.689434477023045\\
3	0.685767118215703\\
5	0.678574330038491\\
9	0.664735629463136\\
17	0.639085575356749\\
33	0.594737613357\\
65	0.52662104063738\\
129	0.438434750537541\\
257	0.344778349422547\\
513	0.262010407153017\\
1025	0.198100233545114\\
2049	0.151787825155577\\
4097	0.117974259765671\\
8193	0.0920768397459965\\
16385	0.071646922638925\\
};\label{fig:obj_data2_rho01_DCA}
\addplot [color=mycolor3,  dashed ,line width=3.0pt, forget plot]
  table[row sep=crcr]{%
0.0211702938436786	0.693147180559993\\
0.0423405876873571	0.693088567612289\\
0.0635108815310357	0.693016403705638\\
0.105851469218393	0.692893461319207\\
0.190532644593107	0.692565447139298\\
0.359894995342535	0.692077629099954\\
0.698619696841392	0.690736514591579\\
1.37606909983911	0.688461918452251\\
2.73096790583453	0.683615819615189\\
5.44076551782539	0.674247003442057\\
10.8603607418071	0.656504709972096\\
21.6995511897705	0.622822320463445\\
43.3779320856974	0.568522868035974\\
86.734693877551	0.490868862040112\\
173.448217461258	0.39774002392606\\
346.875264628673	0.306727872370331\\
693.729358963502	0.231786033574056\\
1387.43754763316	0.175870657329895\\
2774.85392497248	0.135732707503978\\
5549.68667965111	0.105826149865019\\
};\label{fig:obj_data2_rho01_BDCA}

\nextgroupplot[scale only axis,
xmode=log,
xmin=0.01,
xmax=10000,
xminorticks=true,
ymode=log,
ymin=1e-6,
ymax=1e-2,
ytick = {1e-2,1e-3,1e-4,1e-5,1e-6},
ylabel style={font=\footnotesize\color{white!15!black}},
xlabel style={font=\footnotesize\color{white!15!black}},
ylabel={gap},
xlabel={epoch},
name = figure1,
axis background/.style={fill=white},
title style={font=\footnotesize\bfseries},
xmajorgrids,
xminorgrids,
ymajorgrids,
yminorgrids,
]
\addplot [color=mycolor1,  line width=3.0pt, forget plot]
  table[row sep=crcr]{%
1	0.00186098005326422\\
2	0.00183967168157489\\
3	0.00181579673612797\\
5	0.00176917529347744\\
9	0.00168034835383516\\
17	0.00151942501857652\\
33	0.00125591499477179\\
65	0.000897350352642629\\
129	0.000528463544956585\\
257	0.000254563829982507\\
513	0.000103848587035105\\
1025	3.82990448618677e-05\\
2049	1.37912213148875e-05\\
4097	5.14920517296363e-06\\
8193	2.01955740192118e-06\\
16385	7.91986700010258e-07\\
};\label{fig:gap_data2_rho01_DCA}
\addplot [color=mycolor3,  dashed ,line width=3.0pt, forget plot]
  table[row sep=crcr]{%
0.0211702938436786	0.00186095057779821\\
0.0423405876873571	0.00186082053216092\\
0.0635108815310357	0.0018606570523583\\
0.105851469218393	0.00186038154804063\\
0.190532644593107	0.00185843984802961\\
0.359894995342535	0.00185580155693094\\
0.698619696841392	0.00184784020943484\\
1.37606909983911	0.00183489987996713\\
2.73096790583453	0.00180303653931493\\
5.44076551782539	0.00174350511098549\\
10.8603607418071	0.00163088442014247\\
21.6995511897705	0.00141987617085248\\
43.3779320856974	0.00111008455824181\\
86.734693877551	0.000734751693013084\\
173.448217461258	0.000395370757258935\\
346.875264628673	0.000175338119542422\\
693.729358963502	6.78363663275143e-05\\
1387.43754763316	2.44539220708194e-05\\
2774.85392497248	8.90668678561313e-06\\
5549.68667965111	3.4032221469298e-06\\
};\label{fig:gap_data2_rho01_BDCA}
    \end{groupplot}

        \node[anchor=north east, draw = black, line width=0.5pt, fill=white, font=\footnotesize]  
        (legend) at ([shift={(-3.1cm,+3.1cm)}] figure1.north east) 
        {
            \begin{tabular}{l l}
                DCA  & \ref*{fig:gap_data2_rho01_DCA} \\
                \algo & \ref*{fig:gap_data2_rho01_BDCA} 
            \end{tabular}
        };
        
\end{tikzpicture}

\caption{Comparison of DC and \algo for logistic regression with nonconvex sparsity regularization in \eqref{eqn:classification-reg-lK}, in terms of the objective value and the gap (as defined in \Cref{def:gap}), on \textsc{Rcv1} and \textsc{Real-sim} datasets, with $\rho = 0.1$. Epoch represents iteration $\times$ block-size$/m$.}
 \label{fig:logistic_rho1}
\end{figure*} 

\begin{figure*}[!t]
    \centering
    \begin{tikzpicture}
    
    \begin{groupplot}[group style={group size=2 by 2 , horizontal sep = 2cm , vertical sep = 2.0cm},        
        width = 0.36\textwidth, 
        height = 3.5cm,            
        name = {figure_BDCA}
        grid = both, 
        grid style = {gray!30},        
        tick label style = {font=\scriptsize}, 
        name = figure1,
        ]

\nextgroupplot[
scale only axis,
xmode=log,
xmin=0.01,
xmax=10000,
xminorticks=true,
ymode=linear,
ymin=0.0,
ymax=0.8,
ytick = {0,0.1,0.2,0.3,0.4,0.5,0.6,0.7,0.8},
yminorticks=true,
ylabel style={font=\footnotesize\color{white!15!black}},
xlabel style={font=\footnotesize\color{white!15!black}},
ylabel={objective},
xlabel={epoch},
axis background/.style={fill=white},
title style={font=\bfseries},
title style={font=\footnotesize\bfseries},
title={(\textsc{Real-sim}, $\rho = 10$)},
title style={at={(axis description cs:1.1,1)},anchor=south},
xmajorgrids,
xminorgrids,
ymajorgrids,
yminorgrids]

\addplot [color=mycolor1, line width=3.0pt, forget plot]
  table[row sep=crcr]{%
1	0.693147180560083\\
2	0.690885418911668\\
3	0.688410455096051\\
5	0.683572604803572\\
9	0.674326025868702\\
17	0.657404317629549\\
33	0.628829965234622\\
65	0.586622584610179\\
129	0.534336113012632\\
257	0.479225393409811\\
513	0.426640822126174\\
1025	0.379654053026687\\
2049	0.341803729525188\\
4097	0.311910232962644\\
8193	0.290330475023049\\
16385	0.276022418612652\\
};\label{fig:obj_data1_rho10_DCA}
\addplot [color=mycolor3,  dashed ,line width=3.0pt, forget plot]
  table[row sep=crcr]{%
0.0477122000095424	0.693147180560083\\
0.0954244000190849	0.693087978771515\\
0.143136600028627	0.692850442750387\\
0.238561000047712	0.692543735699445\\
0.429409800085882	0.692213154224986\\
0.811107400162221	0.691250522848759\\
1.5745026003149	0.689527585634539\\
3.10129300062026	0.68611561654562\\
6.15487380123097	0.679067827158948\\
12.2620354024524	0.664667243452572\\
24.4763586048953	0.640772956048905\\
48.905005009781	0.60377766920912\\
97.7622978195525	0.554894021049831\\
195.476883439095	0.499927687900059\\
390.906054678181	0.445976032042616\\
781.764397156353	0.396284486502816\\
1563.4810821127	0.353166596036188\\
3126.91445202538	0.321568107961386\\
6253.78119185076	0.295504308832234\\
12507.5146715015	0.279288865773078\\
};\label{fig:obj_data1_rho10_BDCA}

\nextgroupplot[
scale only axis,
xmode=log,
xmin=0.01,
xmax=16385,
xminorticks=true,
ymode=log,
ymin=1e-07,
ymax=1e-02,
ytick = {1e-7,1e-6,1e-5,1e-4,1e-3,1e-2},
yminorticks=true,
ymajorticks=true,
ylabel style={font=\footnotesize\color{white!15!black}},
xlabel style={font=\footnotesize\color{white!15!black}},
ylabel={gap},
xlabel={epoch},
axis background/.style={fill=white},
title style={font=\bfseries},
xmajorgrids,
xminorgrids,
ymajorgrids,
]
\addplot [color=mycolor1, line width=3.0pt, forget plot]
  table[row sep=crcr]{%
1	0.00102880069230245\\
2	0.00124222307788568\\
3	0.00122334589173974\\
5	0.0011866063484041\\
9	0.00111722530616384\\
17	0.000993420800884483\\
33	0.000796414206511045\\
65	0.000543242034914753\\
129	0.000308101059575948\\
257	0.000153888683765561\\
513	7.03974345169543e-05\\
1025	3.04882456679763e-05\\
2049	1.12814412129251e-05\\
4097	4.76937759648952e-06\\
8193	1.49925153064658e-06\\
16385	4.40431103209302e-07\\
};\label{fig:gap_data1_rho10_DCA}
\addplot [color=mycolor3, dashed, line width=3.0pt, forget plot]
  table[row sep=crcr]{%
0.0477122000095424	0.000967668043984701\\
0.0954244000190849	0.00105379337420373\\
0.143136600028627	0.00110018950827344\\
0.238561000047712	0.00112824712656524\\
0.429409800085882	0.00115569749480833\\
0.811107400162221	0.00117416331685187\\
1.5745026003149	0.00117918688660809\\
3.10129300062026	0.00117823523729981\\
6.15487380123097	0.00113131480358932\\
12.2620354024524	0.00103993482017053\\
24.4763586048953	0.00087506854846666\\
48.905005009781	0.000640721441513901\\
97.7622978195525	0.000388487942673938\\
195.476883439095	0.000200787047977192\\
390.906054678181	9.70733800552321e-05\\
781.764397156353	4.387224378346e-05\\
1563.4810821127	1.73632421776208e-05\\
3126.91445202538	7.06293562717719e-06\\
6253.78119185076	2.1656732755887e-06\\
12507.5146715015	6.97194113224712e-07\\
};\label{fig:gap_data1_rho10_BDCA}

\nextgroupplot[scale only axis,
xmode=log,
xmin=0.01,
xmax=10000,
xminorticks=true,
ymode=linear,
ymin=0.0,
ymax=0.8,
ytick = {0,0.1,0.2,0.3,0.4,0.5,0.6,0.7,0.8},
yminorticks=true,
ylabel style={font=\footnotesize\color{white!15!black}},
xlabel style={font=\footnotesize\color{white!15!black}},
ylabel={objective},
xlabel={epoch},
axis background/.style={fill=white},
title style={font=\footnotesize\bfseries},
title={(\textsc{Rcv1}, $\rho = 10$)},
title style={at={(axis description cs:1.1,1)},anchor=south},
name = figure1,
xmajorgrids,
xminorgrids,
ymajorgrids,
yminorgrids,
]

\addplot [color=mycolor1, line width=3.0pt, forget plot]
  table[row sep=crcr]{%
1	0.693147180559993\\
2	0.690526513707113\\
3	0.687645110619211\\
5	0.681997586292014\\
9	0.671136809747999\\
17	0.6510168490754\\
33	0.616220233799142\\
65	0.562576740612639\\
129	0.492233095612722\\
257	0.415337261128724\\
513	0.343244852490016\\
1025	0.281551230671545\\
2049	0.231646970694403\\
4097	0.193604668354915\\
8193	0.167637658131592\\
16385	0.152258925168961\\
};\label{fig:obj_data2_rho10_DCA}
\addplot [color=mycolor3,  dashed ,line width=3.0pt, forget plot]
  table[row sep=crcr]{%
0.0211702938436786	0.693147180559993\\
0.0423405876873571	0.693092680188104\\
0.0635108815310357	0.693047720475252\\
0.105851469218393	0.692957095541142\\
0.190532644593107	0.692678347492574\\
0.359894995342535	0.692182875506091\\
0.698619696841392	0.6909682482011\\
1.37606909983911	0.689196021776917\\
2.73096790583453	0.685420721007045\\
5.44076551782539	0.677091251408485\\
10.8603607418071	0.663267084918574\\
21.6995511897705	0.638127099733751\\
43.3779320856974	0.593861997460091\\
86.734693877551	0.531827953601582\\
173.448217461258	0.457065527998028\\
346.875264628673	0.382035925430848\\
693.729358963502	0.314288469531204\\
1387.43754763316	0.258154091050346\\
2774.85392497248	0.214155290107412\\
5549.68667965111	0.182500869571243\\
};\label{fig:obj_data2_rho10_BDCA}

\nextgroupplot[scale only axis,
xmode=log,
xmin=0.01,
xmax=16385,
xminorticks=true,
ymode=log,
ymin=1e-07,
ymax=0.01,
ytick = {1e-7,1e-6,1e-5,1e-4,1e-3,1e-2},
yminorticks=true,
ylabel style={font=\footnotesize\color{white!15!black}},
xlabel style={font=\footnotesize\color{white!15!black}},
ylabel={gap},
xlabel={epoch},
axis background/.style={fill=white},
title style={font=\bfseries},
xmajorgrids,
xminorgrids,
ymajorgrids,
]
\addplot [color=mycolor1, line width=3.0pt, forget plot]
  table[row sep=crcr]{%
1	0.00118659210625419\\
2	0.00144590382005877\\
3	0.00142590267618056\\
5	0.00138879564396586\\
9	0.00131840404216084\\
17	0.00119164260951469\\
33	0.000986254543636575\\
65	0.00070970790980329\\
129	0.000425995586390902\\
257	0.000214263735478069\\
513	9.44617174321692e-05\\
1025	3.95463922137705e-05\\
2049	1.54079760313016e-05\\
4097	6.00141144974304e-06\\
8193	1.67137891343335e-06\\
16385	5.12982794225276e-07\\
};\label{fig:gap_data2_rho10_DCA}
\addplot [color=mycolor3, dashed,line width=3.0pt, forget plot]
  table[row sep=crcr]{%
0.0211702938436786	0.00118586968820007\\
0.0423405876873571	0.00119989066622999\\
0.0635108815310357	0.00121430907175362\\
0.105851469218393	0.00123557568065341\\
0.190532644593107	0.00127993275449892\\
0.359894995342535	0.0013249856401374\\
0.698619696841392	0.00137239248449816\\
1.37606909983911	0.00138741683133512\\
2.73096790583453	0.00138402548990467\\
5.44076551782539	0.0013411076597191\\
10.8603607418071	0.0012590003983805\\
21.6995511897705	0.00110732832939169\\
43.3779320856974	0.000856554300090873\\
86.734693877551	0.000573351940910601\\
173.448217461258	0.00031749458287695\\
346.875264628673	0.000150512360160917\\
693.729358963502	6.40728515036951e-05\\
1387.43754763316	2.63477960486817e-05\\
2774.85392497248	9.64424866825655e-06\\
5549.68667965111	3.44036777760957e-06\\
};\label{fig:gap_data2_rho10_BDCA}
    \end{groupplot}

        \node[anchor=north east, draw = black, line width=0.5pt, fill=white, font=\footnotesize]  
        (legend) at ([shift={(4.7cm,+3.1cm)}] figure1.north east) 
        {
            \begin{tabular}{l l}
                DCA  & \ref*{fig:gap_data2_rho10_DCA} \\
                \algo & \ref*{fig:gap_data2_rho10_BDCA} 
            \end{tabular}
        };
        
\end{tikzpicture}

\caption{Comparison of DC and \algo for logistic regression with nonconvex sparsity regularization in \eqref{eqn:classification-reg-lK}, in terms of the objective value and the gap (as defined in \Cref{def:gap}), on \textsc{Rcv1} and \textsc{Real-sim} datasets, with $\rho = 10$. Epoch represents iteration $\times$ block-size$/m$.}
\label{fig:logistic_rho10}
\end{figure*} 
 
\subsection{Nonconvex Quadratic Programming with Negative $\ell_1$ Regularization} 

The second experiment examines nonconvex quadratic programming with a negative $\ell_1$
norm regularization term to promote binary solutions, formulated as follows:
\begin{equation}\label{eqn:qubo}
    \min_{\|\bm x\|_{\infty}\leq 1} ~~ \bm x^T\bm Q \bm x -\lambda \|\bm x\|_1,
\end{equation}
where $\bm Q\in\mathbb{R}^{m\times m }$ is the symmetric (not necessarily positive semidefinite) cost matrix. 
We conduct experiments using the Gset benchmark datasets \cite{Gset}, setting the negative adjacency matrix as the cost matrix $\bm Q$. 
We set regularization parameter $\lambda = \|\bm Q\|_F/\sqrt{n}$ {, where $\|\cdot\|_F$ represents the Frobenius norm}. 

In this formulation, we first consider the following decomposition: $f(\bm x)= \bm x^T\bm Q \bm x,~g(\bm x)=0,~ h(\bm x)=\lambda \|\bm x\|_1,$ and $\mathcal{M}=[-1,1]^m$ is the separable constraint set. 
It is important to note that $f(\bm x)$ is nonconvex, but this does not affect the validity of our guarantees. 
It is easy to verify that our guarantees hold for any smooth $f(\bm x)$ as long as the subproblems can be solved globally. 
While nonconvex $f(\bm x)$ makes the subproblems nonconvex, solving them globally along a single coordinate is tractable, as demonstrated in this example.

For DCA, the subproblem becomes:
\begin{equation*}
    \min_{\|\bm x\|_{\infty}\leq 1} ~~ \bm x^T\bm Q \bm x -\lambda \langle \text{sign}(\bm x^k),\bm x - \bm x^k \rangle,
\end{equation*}
which is generally NP-hard. However, using \algo with block size 1, the subproblem simplifies to
\begin{equation*}
    \min_{x_i \in [-1,1]} ~~ (\bar{\bm x}_i^k + \bm x_i)^T\bm Q (\bar{\bm x}_i^k + \bm x_i) -\lambda \langle \text{sign}(\bm x^k),\bar{\bm x}_i^k + \bm x_i - \bm x^k \rangle.
\end{equation*}
Expanding this expression gives
\begin{equation*}
    x_i^{k+1} = \argmin_{x_i \in [-1,1]}~~ ax_i^2 + bx_i + c
\end{equation*}
for scalars $a = q_{ii}$, $b = 2\bm q_{i}^\top\bar{\bm x}_i^k - \lambda \text{sign}(x_i^k)$, and $c = (\bar{\bm x}_i^k)^T \bm Q \bar{\bm x}_i^k + \lambda \text{sign}(x_i^k)x_i^k$, where $q_{ii}$ denotes the $i$th diagonal entry of $\bm{Q}$ and $\bm{q}_i$ is the $i$th column.
This expression can be efficiently solved, highlighting the practical advantages of \algo. Even in scenarios where DCA involves NP-hard subproblems, \algo remains efficient. 

 {Because the DCA subproblem corresponding to the above decomposition is intractable, we instead apply DCA using a different splitting. Specifically, we decompose }$\bm Q = \bm Q_P + \bm Q_N$ where $\bm Q_P$ and $\bm Q_N$ represent positive and negative semidefinite components of $\bm Q$, respectively. 
Under this decomposition, one could redefine $f(\bm x) = \bm x^T\bm Q_P\bm x$ and $h(\bm x) = \lambda \|\bm x\|_1-\bm x^T\bm Q_N\bm x$. 
However, this approach requires eigendecomposition of $\bm Q$, which can be computationally expensive when $m$ is large. 
Moreover, $\bm Q_P$ and $\bm Q_N$ may lose useful structural properties. For example, when $\bm Q$ is sparse, $\bm Q_P$ and $\bm Q_N$ may be dense.

We also compare with the performance of RCSD from \cite{deng2020efficiency}. 
Notably, RCSD can be viewed as a special case of BDCA applied to the following reformulation of the problem:
\begin{equation*}
    \min_{\|\bm x\|_{\infty}\leq 1} ~~  \frac{L}{2}\|\bm x\|_2^2-\lambda \|\bm x\|_1 -(\frac{L}{2}\|\bm x\|_2^2 - \bm x^T\bm Q \bm x),
\end{equation*}
where $L = 2\|\bm Q\|$ is the smoothness constant. 
We compared DCA, \algo, and RCSD in terms of objective values at the local solutions found by the methods, and the computational time needed to reach such solution. 

Both algorithms are initialized from the same random starting point where entries are drawn from a Gaussian distribution, and run until the gap reaches $10^{-6}$. 
The results are summarized in \Cref{table:qubo}. 
Specifically, in 45 cases, \algo finds a better local optimum than DCA, while DCA performs better only in 5 datasets. 
Also, in 31 cases \algo outperforms RCSD and for 11 datasets RCSD finds a better solution. 
In most cases, \algo is faster than both DCA and RCSD. 

\begin{table}
    \centering
    \caption{Comparison of \algo, DCA, and RCSD performance in terms of objective value and computation time in nonconvex quadratic programming with negative $\ell_1$ regularization~\eqref{eqn:qubo}.}
    \begin{adjustbox}{height = 0.68\columnwidth,width=\columnwidth,center}
    \begin{tabular}{|c|c|c|c|c|c|c|}
    \hline
        Dataset & Obj(BDCA) & Obj(DCA) & Obj(RCSD) & Time(BDCA) & Time(DCA) & Time(RCSD) \\ \hline
        G1 & -43891.1 & -43891.1 & -43891.1 & 2.02 & 2.37 & 1.19 \\ \hline
         G2  & -43891.1 & -43891.1 & -43891.1 & 2.01 & 2.59 & 1.15 \\ \hline
         G3  & -43891.1 & -43891.1 & -43891.1 & 2.11 & 2.32 & 0.98 \\ \hline
         G4  & -43891.1 & -43891.1 & -43891.1 & 2.1 & 2.53 & 1.06 \\ \hline
         G5  & -43891.1 & -43891.1 & -43891.1 & 1.55 & 2.05 & 1.38 \\ \hline
         G6  & -11703.1 & -11243.1 & -11363.1 & 2.06 & 5.18 & 4.3 \\ \hline
         G7  & -11635.1 & -11787.1 & -11767.1 & 1.68 & 4.92 & 8.88 \\ \hline
         G8  & -11247.1 & -11691.1 & -11659.1 & 1.74 & 9.81 & 10.11 \\ \hline
         G9  & -11843.1 & -11299.1 & -11683.1 & 2.11 & 4.98 & 9.64 \\ \hline
         G10  & -12107.1 & -11291.1 & -11643.1 & 2.87 & 4.12 & 5.13 \\ \hline
         G11  & -3212 & -3204 & -3244 & 0.63 & 3.01 & 0.62 \\ \hline
         G12  & -3256 & -3184 & -3184 & 0.58 & 2.84 & 0.5 \\ \hline
         G13  & -3260 & -3132 & -3196 & 0.6 & 3.05 & 0.56 \\ \hline
         G14  & -12128.51 & -12128.51 & -12128.51 & 1.27 & 10.31 & 1.99 \\ \hline
         G15  & -7356.86 & -6604.86 & -12052.86 & 2.13 & 15.16 & 4.05 \\ \hline
         G16  & -6650.08 & -6438.08 & -6610.08 & 1.1 & 9.9 & 4.16 \\ \hline
         G17  & -12066.62 & -7486.62 & -7038.62 & 1.46 & 17.42 & 2.38 \\ \hline
         G18  & -5556.51 & -5680.51 & -5596.51 & 0.93 & 14.47 & 2 \\ \hline
         G19  & -5576.86 & -5564.86 & -5568.86 & 1.12 & 13.37 & 1.91 \\ \hline
         G20  & -5538.08 & -5346.08 & -5482.08 & 1.25 & 7.97 & 1.99 \\ \hline
         G21  & -5626.62 & -5438.62 & -5430.62 & 1.23 & 8.34 & 2.88 \\ \hline
         G22  & -48922.04 & -48922.04 & -48922.04 & 6.36 & 60.94 & 8.47 \\ \hline
         G23  & -48922.04 & -48922.04 & -48922.04 & 6.24 & 37.38 & 5.65 \\ \hline
         G24  & -48922.04 & -18238.04 & -48922.04 & 4.77 & 41.08 & 9.57 \\ \hline
         G25  & -48922.04 & -17910.04 & -48922.04 & 4.7 & 40.1 & 26.15 \\ \hline
         G26  & -48922.04 & -48922.04 & -48922.04 & 4.79 & 86.45 & 8.11 \\ \hline
         G27  & -18566.04 & -18054.04 & -18042.04 & 5 & 35.58 & 31.21 \\ \hline
         G28  & -18594.04 & -18154.04 & -18386.04 & 4.64 & 41.69 & 16.08 \\ \hline
         G29  & -18442.04 & -17842.04 & -18054.04 & 4.94 & 50.26 & 30.09 \\ \hline
         G30  & -18350.04 & -17958.04 & -18702.04 & 4.65 & 47.92 & 13.92 \\ \hline
         G31  & -18690.04 & -18074.04 & -18690.04 & 4.49 & 45.72 & 11.42 \\ \hline
         G32  & -8172 & -8076 & -8236 & 1.37 & 15.26 & 1.95 \\ \hline
         G33  & -8212 & -7996 & -8140 & 1.13 & 15.67 & 2.3 \\ \hline
         G34  & -8232 & -8024 & -8120 & 1.13 & 16.09 & 2.24 \\ \hline
         G35  & -16999.82 & -30419.82 & -30419.82 & 2.46 & 74.96 & 10.47 \\ \hline
         G36  & -30392.32 & -16928.32 & -30392.32 & 3.54 & 65.26 & 10.91 \\ \hline
         G37  & -30267.86 & -30435.86 & -30435.86 & 4.73 & 63.71 & 14.23 \\ \hline
         G38  & -30422.11 & -30422.11 & -30422.11 & 4.04 & 132.32 & 8.25 \\ \hline
         G39  & -13811.82 & -13551.82 & -13651.82 & 2.33 & 41.84 & 12.16 \\ \hline
         G40  & -13792.32 & -13584.32 & -13912.32 & 2.32 & 56.62 & 16.06 \\ \hline
         G41  & -13915.86 & -13815.86 & -13683.86 & 2.32 & 45.32 & 11.78 \\ \hline
         G42  & -14386.11 & -13806.11 & -13862.11 & 2.36 & 52.1 & 14.3 \\ \hline
         G43  & -24449.9 & -24449.9 & -24449.9 & 0.92 & 4.52 & 1.57 \\ \hline
         G44  & -24449.9 & -24449.9 & -24449.9 & 0.98 & 4.11 & 1.5 \\ \hline
         G45  & -24449.9 & -24449.9 & -24449.9 & 0.96 & 3.78 & 1.51 \\ \hline
         G46  & -24449.9 & -24449.9 & -24449.9 & 0.97 & 3.9 & 1.61 \\ \hline
         G47  & -24449.9 & -24449.9 & -24449.9 & 0.96 & 4.1 & 1.44 \\ \hline
         G48  & -13584 & -13112 & -13472 & 2.96 & 63.35 & 5.12 \\ \hline
         G49  & -13512 & -13184 & -13320 & 2.32 & 213.11 & 5 \\ \hline
    \end{tabular}
    \end{adjustbox}
    \label{table:qubo}
\end{table}

\begin{table}
    \centering
    \begin{adjustbox}{width=\columnwidth,center}
    \begin{tabular}{|c|c|c|c|c|c|c|}
    \hline
        Dataset & Obj(BDCA) & Obj(DCA) & Obj(RCSD) & Time(BDCA) & Time(DCA) & Time(RCSD) \\ \hline
         G50  & -13536 & -13168 & -13456 & 2.96 & 45.9 & 5.56 \\ \hline
         G51  & -15255.73 & -15255.73 & -15255.73 & 0.94 & 18.71 & 1.55 \\ \hline
         G52  & -15271.77 & -9083.77 & -15271.77 & 1.26 & 17.97 & 2.74 \\ \hline
         G53  & -15267.19 & -8131.19 & -8327.19 & 1.25 & 6.76 & 3.78 \\ \hline
         G54  & -15271.77 & -8239.77 & -15271.77 & 1.08 & 8.65 & 3.86 \\ \hline
         G55  & -21687.45 & -21039.45 & -21863.45 & 18.82 & 280.19 & 72.61 \\ \hline
         G56  & -21939.45 & -21055.45 & -21579.45 & 14.63 & 359.52 & 63.93 \\ \hline
         G57  & -20196 & -19796 & -20100 & 7.62 & 936.13 & 11.47 \\ \hline
         G58  & -76335.93 & -76335.93 & -76335.93 & 29.66 & 753.15 & 78.43 \\ \hline
         G59  & -35019.93 & -34219.93 & -34871.93 & 18.88 & 590.34 & 179.67 \\ \hline
         G60  & -30462.26 & -29770.26 & -30386.26 & 26.31 & 836.66 & 166.26 \\ \hline
         G61  & -30094.26 & -29406.26 & -29966.26 & 27.97 & 882.6 & 170.17 \\ \hline
         G62  & -28296 & -27984 & -28232 & 13.11 & 2398.02 & 31.71 \\ \hline
         G63  & -107010.03 & -61763 & -107010.03 & 117.79 & 4218.65 & 354.19 \\ \hline
         G64  & -49222.03 & -48250.03 & -48670.03 & 43.75 & 1535.71 & 365.12 \\ \hline
         G65  & -32372 & -32156 & -32532 & 17.69 & 3468.09 & 32.9 \\ \hline
         G66  & -36720 & -36360 & -36704 & 25.39 & 1011.99 & 55.51 \\ \hline
         G67  & -40428 & -40052 & -40402 & 33.96 & 6298.46 & 63.5 \\ \hline
    \end{tabular}
    \end{adjustbox}
\end{table}
\section{Conclusions}

We investigated a general class of DC problems involving smooth and separable nonsmooth terms, and proposed a randomized block-coordinate DC algorithm (\algo) specifically designed to solve such problems efficiently. 
We established convergence guarantees for \algo with a rate of $\mathcal{O}(n/K)$, where the algorithm performs $K$ iteration over $n$ coordinate blocks. 
Leveraging the connection between the DCA and the EM method, we proposed a  {randomized} block-coordinate EM algorithm (\algoem), where convergence guarantees are directly inherited from the DCA as a special case.
DCA and EM have numerous applications in data science and machine learning; hence, scalable alternatives like \algo and \algoem have the potential for significant impact. 

\subsection*{Acknowledgments}
Hoomaan Maskan and Alp Yurtsever were supported by the Wallenberg AI, Autonomous Systems and Software Program (WASP) funded by the Knut and Alice Wallenberg Foundation. Suvrit Sra acknowledges generous support from the Alexander von Humboldt Foundation.  {We thank the High Performance Computing Center North (HPC2N) at Umeå University for providing computational resources and valuable support during test and performance runs. The computations were enabled by resources provided by the National Academic Infrastructure for Supercomputing in Sweden (NAISS), partially funded by the Swedish Research Council through grant agreement no. 2022-06725.}

\begin{small}
\setlength{\bibsep}{4pt}
\bibliography{references}
\end{small}

\renewcommand{\theequation}{SM\arabic{equation}}
\renewcommand{\thetheorem}{SM\arabic{theorem}}
\renewcommand{\thefigure}{SM\arabic{figure}}
\renewcommand{\thedefinition}{SM\arabic{definition}}
\renewcommand{\thealgorithm}{SM\arabic{algorithm}}

\setcounter{equation}{0}
\setcounter{algorithm}{0}
\setcounter{theorem}{0}
\newpage
\appendix 

\section{Regularized EM}
In \Cref{sec:EM}, we discussed the relationship between the EM algorithm and DCA. 
To maintain simplicity and consistency with the standard EM algorithm, we focused on minimizing the negative log-likelihood function. 
However, it is common practice to solve regularized log-likelihood problems, where a regularization term is introduced to prevent overfitting, improve generalization, or promote specific structures in the solution \cite{li2005regularized,yi2015regularized}. 
In this context, the Regularized EM (REM) algorithm seeks to optimize the penalized negative log-likelihood problem:
\begin{equation} \label{eqn:RML}
    \min_{\bm{\theta} \in \Theta} ~ \hat{\mathcal{L}}(\bm{\theta}) :=  \sum_{\bm{x} \in \mathcal{X}}\sum_{\bm{y}}-\log P(\bm{x},\bm{y}|{\bm{\theta}}) + \lambda R(\bm{\theta}), 
\end{equation}
where $R$ is a convex regularization function with block-separable structure $R(\bm{\theta}) = \sum_{i=1}^n R_i({\theta_i})$, and $\lambda \geq 0$ is the regularization parameter. 
Similarly, we assume $\Theta$ can be decomposed as $\Theta_1 \times \cdots \times \Theta_n$, where the components, such that $\bm{\theta} \in \Theta$ if and only if $\theta_i \in \Theta_i$ for all $i=1,\ldots,n$.  

For exponential family of distributions, as defined in \Cref{def:expo-family}, we can express \eqref{eqn:RML} as a specific instance of the problem template
\begin{equation}\label{eqn:REM_DC}
    \min_{\bm{\theta} \in \Theta} ~ f(\bm{\theta}) + g(\bm{\theta}) - h(\bm{\theta}),
\end{equation}
where $f$ and $h$ are as defined in \eqref{eqn:functions}, and $g(\bm{\theta}) = \lambda R(\bm{\theta})$.

Employing \algo to solve \eqref{eqn:REM_DC} results in Block REM method shown in \Cref{alg:BCREM}. The following corollary, which is a direct consequence of Theorem 4, formulates the convergence behaviour of the Block REM method in terms of the introduced gap function (See \Cref{def:gap}).
\begin{corollary}
    Suppose there exists a sequence ${\bm{\theta}^k}, k =1\ldots,K$ generated by Block REM algorithm. 
    Then, 
        \begin{equation*}
        \min_{k\in\{1,\ldots,K\}} \expect{\gap^L_{\Theta}(\bm{\theta}^k)}
        \leq \frac{n}{K} \left(\hat{\mathcal{L}}(\bm{\theta}^1) - \hat{\mathcal{L}}^\star\right).
    \end{equation*}
\end{corollary}

\begin{algorithm*}[tb]
   \caption{Block REM Algorithm}
   \label{alg:BCREM} 
\begin{algorithmic}
   \STATE {\bfseries Input:} regularization parameter $\lambda$, total iterations $K$, initialize the probability distribution $P$ using vector $\bm{\theta}^{(0)}$ ($P$ belongs to the exponential family), and number of blocks $n$
   \FOR{ $k=1$ {\bfseries to} $K$}
        \STATE Randomly choose $i_k$ in $[1,...,n]$ with uniform distribution
        \STATE Update $ \hat{P}^{k+1}(\bm{y}) =\frac{P(\bm{x},\bm{y}|\bm{\theta}^k)}{\sum_{\bm{y}} P(\bm{x},\bm{y}|\bm{\theta}^k)}$
        \STATE Update ${\theta}_{i_k}^{k+1} = \underset{\theta_{i_k} \in \Theta_{i_k}}{\mathrm{arg\,min}} \sum_{\bm{x}\in\mathcal{X},\bm{y}}-\hat{P}^{k+1}(\bm{y})\log P(\bm{x},\bm{y}|\theta^k_{1},\ldots,\theta_{i_k},\ldots,\theta^k_{n}) + \lambda R_{i_k}(\theta_{i_k})$
        \STATE Set $\bm{\theta}^{k+1} = \bar{\bm{\theta}}_{i_k}^{k}+ \bm{\theta}_{i_k}^{k+1}$
   \ENDFOR
   \STATE {\bfseries Output: $\bm{\theta}^{K+1}$}
\end{algorithmic}
\end{algorithm*}

\section{BDCA Recovers Coordinate (Sub)Gradient Descent}

Randomized coordinate subgradient descent method (RCSD) is a block coordinate proximal type subgradient algorithm for similar DC nonconvex problems as ours (see problem \eqref{eqn:main}).
RCSD assumes an additional assumption on function $g$ to have a simple structure \cite{deng2020efficiency}.
Here, we briefly present RCSD and its convergence guarantee. Later, we attempt to show how \algo recovers RCSD.

RCSD is a block coordinate subgradient method which iteratively updates some random coordinates while keeping the rest of the coordinates fixed (similar to \algo). 
To formally present the method, we need block proximal mapping given by
\begin{equation*}
    \mathcal{P}_i(\bar{x}_i,y_i,\gamma_i) = \argmin_{x\in\mathbb{R}^{m_i}}\left\{ \ip{y_i}{x} +  g_i(x) + \frac{\gamma_i}{2}\|\bar{x}_i-x\|_i^2\right\},\quad \forall y_i\in\mathbb{R}^{m_i}
\end{equation*}
where $\bm \gamma = [\gamma_1,\gamma_2,\ldots,\gamma_n]$, $\|.\|_i$ is the standard Euclidean distance on $\mathbb{R}^{m_i}$, and $\mathcal{P}(\bar{\bm{x}},\bm{y},\bm \gamma) = \sum_{i=1}^n \bm D_i^T \bm D_i\mathcal{P}_i(\bar{x}_i,y_i,\gamma_i)$. The RCSD method is shown in \Cref{alg:RCSD}.

Next, we will present the convergence result of RCSD for uniform sampling strategy. For more detailed description of the results, see \cite{deng2020efficiency}.
\begin{theorem}[Theorem 3, \cite{deng2020efficiency}]\label{theorem:RCSD}
    Suppose the sequence $\bm{x}^k, k=1,\ldots, K$, is generated by \Cref{alg:RCSD}. Assume $\gamma_i = L_i$ for $L_i$ being block wise Lipschitz constants for function $f$. Then,
    \begin{equation*}
        \min_{k\in\{0,\ldots,K\}} \mathbb{E}\Big\{ \|G({\bm{x}}^{k},\nabla f(\bm{x}^k) - \bm v^k,\bm \gamma)\|^2\Big\}\leq \frac{2 L_{max} n (\phi(\bm{x}^1)-\phi^*)}{K},
    \end{equation*}
    where $G({\bm{x}}^{k},\nabla f(\bm{x}^k) - \bm v^k,\bm\gamma) = \sum_{i=1}^n \bm D_i^T \bm D_i G_i({x}_{i}^{k},\nabla_i f(\bm{x}^k) - v_i^k,\gamma_i)$, $L_{max} = \max_{i\in\{1,\ldots,n\}} L_i$, and
    \[G_i({x}_{i},\nabla_i f(\bm{x}) - v_i,\gamma_i) = \gamma_i(x_i- \mathcal{P}_{i}({x}_{i},\nabla_{i}f(\bm{x}) - v_{i},\gamma_{i})).\]
\end{theorem}
Now, we will show how \algo relates to RCSD. This result is based on the smoothness of $f$.
\begin{proposition}
    \algo covers RCSD with uniform sampling strategy as a special instance. 
\end{proposition}
\begin{proof}
    Due to smoothness of $f$ we can write the objective function as 
    \begin{equation*}
        \phi(\bm{x}) = \frac{L_{max}}{2}\|\bm{x}\|^2 + g(\bm{x})  - (h(\bm{x}) + \frac{L_{max}}{2}\|\bm{x}\|^2 - f(\bm{x})). 
    \end{equation*}
    Take $\bm v_{i_k}^k \in \partial_{i_k} h(\bm{x}^k)$. 
    This gives the \algo subproblem as
    \begin{equation}\label{RCSD}
        \begin{aligned}
            x_{i_k}^{k+1} &\in \argmin_{x_{i_k} \in \mathcal{M}_{i_k}} \frac{L_{max}}{2}\| x_{i_k}\|^2 + g_{i_k}(x_{i_k})
            -\langle \bm v_{i_k}^k - \nabla_{i_k} f(\bm{x}^k) + L_{max} x^k_{i_k},x_{i_k} \rangle  \\
             &= \argmin_{x_{i_k} \in \mathcal{M}_{i_k}} \langle  \nabla_{i_k} f(\bm{x}^k) - \bm v_{i_k}^k , x_{i_k}\rangle
             + g_{i_k}(x_{i_k}) + \frac{L_{max}}{2}\| x_{i_k} - x_{i_k}^k\|^2\\
             &= \argmin_{x_{i_k} \in \mathcal{M}_{i_k}} \langle  \nabla_{i_k} f(\bm{x}^k)  {-} \bm v_{i_k}^k , x_{i_k} - x_{i_k}^k \rangle 
             + g_{i_k}(x_{i_k}) + \frac{L_{max}}{2}\| x_{i_k} - x_{i_k}^k\|^2 \\
             &= \argmin_{x_{i_k} \in \mathcal{M}_{i_k}} \frac{L_{max}}{2}\left[ \frac{2}{L_{max}}\langle \nabla_{i_k} f(\bm{x}^k) - \bm v_{i_k}^k, x_{i_k} - x_{i_k}^k \rangle \right]
             + g_{i_k}(x_{i_k}) + \frac{L_{max}}{2}\| x_{i_k} - x_{i_k}^k\|^2 \\
             &= \argmin_{x_{i_k} \in \mathcal{M}_{i_k}}  g_{i_k}(x_{i_k}) 
             + \frac{L_{max}}{2}\| x_{i_k} - (x_{i_k}^k-\frac{1}{L_{max}}(\nabla_{i_k} f(\bm{x}^k) - \bm v_{i_k}^k))\|^2 \\
             &= \prox_{\frac{1}{L_{max}}g_{i_k}+I_{\mathcal{M}_{i_k}}}(x_{i_k}^k-\frac{1}{L_{max}}(\nabla_{i_k} f(\bm{x}^k) - \bm v_{i_k}^k))\nonumber
        \end{aligned}
    \end{equation}
    which is the block proximal operator with constant $1/L_{max}$ with $I_{\mathcal{C}}$ denoting the indicator function on the set $\mathcal{C}$. Application of this update for each randomly selected block $i_k$ in \algo gives RCSD. 
\end{proof} 

\begin{algorithm*}[tb]
   \caption{RCSD Algorithm}
   \label{alg:RCSD} 
    \begin{algorithmic}
       \STATE {\bfseries Input:} $\bm{x}^{(0)}$, total iterations $K$, coefficients $\bm \gamma$, and number of block $n$
       \FOR{ $k=1$ {\bfseries to} $K$}
            \STATE Randomly choose $i_k$ in $[1,...,n]$ with uniform distribution
            \STATE Compute $\nabla_{i_k}f(\bm{x}^k)$ and $\bm v^k\in\partial h(\bm{x}^k)$
            \STATE Update ${x}_{i_k}^{k+1} = \mathcal{P}_{i_k}({x}_{i_k}^{k},\nabla_{i_k}f(\bm{x}^k) - v_{i_k}^k,\gamma_{i_k}) $
            \STATE Set $\bm{x}^{k+1} = \bar{\bm{x}}_{i_k}^{k}+ \bm{x}_{i_k}^{k+1}$
       \ENDFOR
        \STATE {\bfseries Output: $\bm{x}^{K+1}$}
    \end{algorithmic}
\end{algorithm*}

\section{Stationarity in Nonsmooth Optimization}

In this section, we illustrate through simple examples that two common measures for first-order stationarity in smooth optimization, namely the distance measure
\begin{equation*}
    \mathrm{dist}(0,\nabla f(\bm{y}) + \partial g(\bm{y}) - \nabla h(\bm{y}) + \mathcal{N}_{\mathcal{M}}(\bm{y})),
\end{equation*}
and the Frank-Wolfe gap (but with a subgradient $\bm{u} \in \partial g(\bm{y})$) 
\begin{equation*}
    \max_{\bm{x} \in \mathcal{M}} ~  \ip{\nabla f(\bm{y}) + \bm{u} - \nabla h(\bm{y})}{ \bm{y} - \bm{x}}   
\end{equation*}
are intractable in nonsmooth optimization. 
Specifically, we show that these measures can be discontinuous at the solution and remain lower bounded even when approaching the solution arbitrarily closely. 

Consider the problem where $\mathcal{M}=[-1,1]$, $f(x)=h(x)=0$, and $g = |x|$. 
The problem is convex, and the only first-order stationary point is at $x=0$, which is also the global minimum. 
The subdifferential of $g$ is and the normal cone of $\mathcal{M}$ are given by
\begin{equation*}
    \partial g(x) = 
    \begin{cases}
        \{1\} & \text{if $x > 0$}, \\
        \{-1\} & \text{if $x < 0$}, \\
        [-1,1] & \text{if $x = 0$},
    \end{cases}
    \qquad
    \mathcal{N}_{\mathcal{M}}(x) = 
    \begin{cases}
        \{0\} & \text{if $x \in (-1,1)$}, \\
        (-\infty, 0] & \text{if $x = -1$}, \\
        [0,+\infty) & \text{if $x = 1$}.
    \end{cases}
    \qquad\qquad
\end{equation*}
Then, the distance measure becomes
\begin{equation*}
    \mathrm{dist}(0,\partial g(y) + \mathcal{N}_{\mathcal{M}}(y)) 
    = \min_{\substack{u \in \partial g(y)\\z\in\mathcal{N}_{\mathcal{M}}(y)}} |u+z-0| = 
        \begin{cases}
            0 & \text{if $y=0$}, \\
            1 & \text{otherwise},
        \end{cases}
\end{equation*}
indicating that the distance is $0$ only at the global optimum and remains $1$ at all other points, even when arbitrarily close to the global optimum. 

Similarly, the Frank-Wolfe gap, even minimized over the subdifferential set, becomes
\begin{equation}
    \min_{u \in \partial g(y)} \, \max_{x \in [-1,1]} ~  \ip{u}{y - x} = 
    \begin{cases}
        0 & \text{if $y = 0$}, \\ 
        y+1 & \text{if $y > 0$}, \\
        1-y & \text{if $y < 0$},
    \end{cases}
\end{equation}
which shows that, except at $y=0$, the Frank-Wolfe gap is larger than $1$ at all other points, even when arbitrarily close to the global optimum. 

 {Now, if we use the gap function as in \Cref{def:gap}, we have:
\begin{equation*}
    \gap_{\mathcal{M}}^0 (y)= \max_{x\in[-1,1]} \left\{ |y| - |x| \right\} = |y|,
\end{equation*}
which is always non-negative and gets arbitrarily small as we approach the solution at $y=0$.}

Notably, since the examples are convex, these findings also extend to convex nonsmooth optimization problems.

\section{Proof of \Cref*{rem:gap_prox}}

We start by showing that $\gap_{\mathcal{M}}^L(\bm{y})=0$ if 
\begin{equation*}
    \bm{y} = \prox_{\frac{1}{L} g+I_{\mathcal{M}}} \Big((\bm{y}-\tfrac{1}{L}\big(\nabla f(\bm{y}) - \nabla h(\bm{y}) \big) \Big).
\end{equation*}
Using the definition of the proximal map we have
\begin{align*}
   \bm{y} &= \argmin_{\bm{x}\in\mathcal{M}}~~ g(\bm{x}) + \frac{L}{2}\|\bm{x}-(\bm{y}-\tfrac{1}{L}\big(\nabla f(\bm{y}) - \nabla h(\bm{y}) \big) )\|^2\nonumber\\
   &= \argmin_{\bm{x}\in\mathcal{M}}~~ g(\bm{x}) + \frac{L}{2}\|\bm{x} - \bm{y}\|^2 + \langle \bm{x} - \bm{y} , \nabla f(\bm{y}) - \nabla h(\bm{y})\rangle \nonumber\\
   &= \argmax_{\bm{x}\in\mathcal{M}}~~ -g(\bm{x}) - \frac{L}{2}\|\bm{x} - \bm{y}\|^2 + \langle \bm{y} - \bm{x} , \nabla f(\bm{y}) - \nabla h(\bm{y})\rangle .
\end{align*}
From this formulation, it follows that $\gap_{\mathcal{M}}^L(\bm{y})=0$. 

Next, we show that $\bm{y}$ is the fixed point of the proximal mapping  {when} $\gap_{\mathcal{M}}^L=0$. We start by the reformulating the definition of the gap function:
\begin{equation*}
    \gap^L_{\mathcal{M}}(\bm{y}) = L \max_{\bm{x} \in \mathcal{M}}  \bigg\{ \ip{\bm{x} + \frac{1}{L}\big(\nabla f(\bm{y})  - \nabla h(\bm{y})\big) - \bm{y}}{ \bm{y} - \bm{x}} + \frac{1}{L}(g(\bm{y}) - g(\bm{x})) + \frac{1}{2} \norm{\bm{x} - \bm{y}}^2 \bigg\}.
\end{equation*}
 {Suppose $\gap_{\mathcal{M}}^L =  {\epsilon}$. Then, by definition:}
\begin{equation} \label{eqn:prf_rem2_3}
    \ip{\bm{x} + \frac{1}{L}\big(\nabla f(\bm{y})  - \nabla h(\bm{y})\big) - \bm{y}}{ \bm{y} - \bm{x}} + \frac{1}{L}g(\bm{y}) - \frac{1}{L}g(\bm{x}) + \frac{1}{2} \norm{\bm{x} - \bm{y}}^2 \leq  {\frac{\epsilon}{L}}, 
\end{equation}
 {for all $\bm{x} \in \mathcal{M}$.} 
Consider this with $\bm{x} = \prox_{\frac{1}{L}g+I_{\mathbf{M}}}(\bm{u})$ and $\bm{u} = \bm{y}-\tfrac{1}{L}\big(\nabla f(\bm{y}) - \nabla h(\bm{y}) \big)$. 
Using non-expansiveness of the proximal operator we have
\begin{align}\label{eqn:prf_rem2_4}
      \ip{\bm{x} - \bm{u}}{\bm{y} - \bm{x}} + \frac{1}{L}g(\bm{y}) - \frac{1}{L} g(\bm{x}) &\geq 0.
\end{align}
Utilizing \eqref{eqn:prf_rem2_4} in \eqref{eqn:prf_rem2_3},  {we obtain the desired bound:
\begin{equation*}
    \frac{1}{2} \norm{ \bm{y} - \prox_{\frac{1}{L}g+I_{\mathbf{M}}}(\bm{u}) }^2 \leq \frac{\epsilon}{L}.
\end{equation*}
In particular, if $\epsilon = 0$, this implies $\bm{y} = \prox_{\frac{1}{L}g+I_{\mathcal{M}}}(\bm{u})$.}

 {
\section{Proof of \Cref*{lem:gap-nonsmooth}}
\label{app:gap_nonsmooth_proof}

    The proof follows the same reasoning as in \Cref{lem:gap}, with the extension that $h$ is allowed to be non-differentiable.

    The first statement is immediate, since the maximized expression vanishes when $\bm{x} = \bm{y}$, and $\bm{y} \in \mathcal{M}$. 

    For the second statement, let us first assume that condition \eqref{eqn:criticality} holds.
    Then, there exist $\bm{v}^\star \in \partial h(\bm y)$ and $\bm u^\star \in \partial g(\bm y)$ such that $ \bm{v}^* -\bm u^* - \nabla f(\bm y) \in \mathcal{N}_{\mathcal{M}}(\bm y)$, or equivalently
    \begin{equation*}
        \langle \nabla f(\bm y) +  \bm u^* -\bm{v}^*  , \bm y - \bm x\rangle \leq 0, \quad \forall \bm x \in \mathcal{M}. 
    \end{equation*}
    Since $g$ is convex, this implies
    \begin{align*}
      & \langle \nabla f(\bm y) -\bm{v}^*  , \bm y - \bm x\rangle +  g(\bm y) - g(\bm x)\leq 0\quad \forall \bm x \in \mathcal{M},\\
       \implies& \langle \nabla f(\bm y) -\bm{v}^*  , \bm y - \bm x\rangle +  g(\bm y) - g(\bm x) - \frac{L}{2}\|\bm x - \bm y\|^2 \leq 0\quad \forall \bm x \in \mathcal{M}.
    \end{align*}
    Since this inequality holds for all $\bm x \in \mathcal{M}$, we can maximize the left-hand side over $\bm{x}$. 
    Moreover, since it holds for at least one subgradient $\bm{v}^* \in \partial h(\bm y)$, we can minimize over the subdifferential set:
    \begin{equation*}
        \min_{\bm u \in \partial g(\bm y)} \max_{\bm x \in \mathcal{M}} ~ \Big\{  \langle \nabla f(\bm y) -\bm{v}  , \bm y - \bm x\rangle +  g(\bm y) - g(\bm x) - \frac{L}{2}\|\bm x - \bm y\|^2 \Big\} \leq 0.
    \end{equation*}
    Finally, by Von Neumann’s minimax theorem, we can interchange the order of $\min$ and $\max$ since both $\partial g(\bm y)$ and $\mathcal{M}$ are convex and compact sets; and the expression inside the braces is concave in $\bm x$ for any fixed $\bm u \in \partial g(\bm y)$, and convex in $\bm{v}$ for any fixed $\bm x \in \mathcal{M}$. 
    After this change, we obtain $\gap^L_{\mathcal{M}}(\bm{y}) \leq 0$. 
    Since the gap function is nonnegative by definition, we conclude that $\gap^L_{\mathcal{M}}(\bm{y}) = 0$. 

    We now prove the reverse direction. Suppose that $\gap^L_{\mathcal{M}}(\bm{y}) = 0$. 
    Then, there exists a subgradient $\bm{v}^\star \in \partial h(\bm y)$ such that
    \begin{equation*}
        \ip{\nabla f(\bm{y})  - \bm{v}^\star}{ \bm{x} - \bm{y}} + g(\bm{x}) - g(\bm{y}) + \frac{L}{2} \norm{\bm{x} - \bm{y}}^2 \geq 0, \quad \forall \bm x\in \mathcal{M}.
    \end{equation*}
    Consider $\bm x=\bm y+\alpha \bm d$ for an arbitrary feasible direction $\bm d$ (unit norm) and step-size $\alpha > 0$. Then, 
    \begin{equation*}
        \alpha \ip{\nabla f(\bm{y})  - \bm{v}^\star}{\bm d} + g(\bm y+\alpha \bm d) - g(\bm y) + \frac{L}{2}\alpha^2\geq 0, \quad \forall \alpha \bm d:\bm y+\alpha \bm d\in \mathcal{M}.
    \end{equation*}
    Dividing by $\alpha$ and taking the limit as $\alpha \to 0^+$, we obtain
    \begin{equation*}
        \langle\nabla f(\bm y)- \bm{v}^\star,\bm d\rangle + \langle \bm u^\star,  \bm d\rangle\geq 0, \quad \forall  \bm d:\lim_{\alpha\rightarrow 0^+}\bm y+\alpha \bm d\in \mathcal{M}.
    \end{equation*}
    for some $\bm u^\star \in \partial g(\bm y)$. 
    Since $\mathcal{M}$ is closed and convex, $\lim_{\alpha\rightarrow 0^+}\bm y+\alpha \bm d\in \mathcal{M}$ for all $\bm d=\bm x-\bm y$ such that $\bm{x} \in \mathcal{M}$.
    Consequently, 
    \begin{equation*}
        \langle \nabla f(\bm y) + \bm u^\star -\bm{v}^\star, \bm x - \bm y\rangle \geq 0, \quad \forall \bm x \in \mathcal{M}, 
    \end{equation*}
    which is equivalent to 
    $-\nabla f(\bm y) - \bm u^\star +\bm{v}^\star \in \mathcal{N}_{\mathcal{M}}(\bm y).$ 
    Therefore, we have condition \eqref{eqn:criticality} satisfied. 
}

\section{Application to One-bit MIMO Signal Recovery} \label{app:Mimo}

In this section, we will use the proposed Block EM method in one-bit MIMO signal recovery problem. First, we describe the problem. Later, simulation results are given to show the performance of our method compared to the conventional EM method.

\subsection{Problem Description}
Consider a MIMO uplink scenario where $N$ transmitters (users) simultaneously send a signal of length $W$ to the Base-Station (BS) which consists of $M$ receiving antennas. 
The signal will undergo a frequency selective fading channel. 
Then, the recieved signal at the BS will have the form
\begin{align}\label{MIMO_prob}
    \boldsymbol R_m = \sum_{n=1}^N \boldsymbol H_{m,n} \boldsymbol \theta_n + \boldsymbol \alpha_m \quad \text{and} \quad \boldsymbol{Y_m} = Q(\boldsymbol{R}_m), \quad \text{for} \quad m=1,\ldots,M,
\end{align}
where $\boldsymbol R_m\in \mathbb C^{W}$ is the received signal at $m$th antenna, $\boldsymbol H_{m,n}\in \mathbb C^{W\times W} $ denotes the channel circulant matrix between the $n$the user and $m$th receiver, $\boldsymbol \theta_n\in \mathbb C^W$ is the transmitted signal of the $n$th user, and $\alpha_m\in\mathbb C^W$ is the complex additive white Gaussian noise for the $m$th receiver with mean $\boldsymbol 0$ and covariance matrix $\sigma_{\alpha}^2 \boldsymbol I$. Note that each $H_{m,n}$ is constructed from the channel impulse response $h_{m,n}=\{h_{m,n}^0,h_{m,n}^1,\ldots,h_{m,n}^{L}, 0\ldots, 0\}\in\mathbb C^W$ between the $n$th user and $m$th receiver with $L$ being the number of channel taps. The operator $Q(.)$ serves as the sign operator on the real and imaginary parts of its argument. Throughout this section we assume the transmitted signal to be modulated with 16 QAM and $W\gg L$. This means that $\bm{\theta} \in \{-3,-1,+1,+3\}$. It is possible to recover such signal with EM method as \cite{8491219}. Here, the aim is to use the Block EM method to recover the modulated signal $\boldsymbol \theta$. To do so, we represent \eqref{MIMO_prob} as 
\begin{align*}
    \boldsymbol R = \mathcal{H}_{cir}\boldsymbol \theta + \boldsymbol \alpha \qquad \text{and} \qquad \boldsymbol{Y} = Q(\boldsymbol R)
\end{align*}
where $\boldsymbol R \in \mathbb C^{MW\times 1},\boldsymbol \alpha\in \mathbb C^{MW\times 1},\boldsymbol{Y}\in \mathbb C^{MW\times 1},\boldsymbol \theta\in \mathbb C^{NW\times 1}$ are the vertically stacked $\boldsymbol R_m,\boldsymbol \alpha_m,\boldsymbol{Y}_m,\boldsymbol \theta_n$ over $m$ and $n$ respectively. $\mathcal{H}_{cir}\in\mathbb C^{MW\times NW}$ is the block-circulant Toeplitz matrix made from $\boldsymbol{H}_{m,n}$'s. According to \eqref{MIMO_prob} we need to form the conditional probability distribution function (PDF) $p(\boldsymbol R|\boldsymbol Y,\boldsymbol \theta)$ (note that we are treating $\bm{y}$ as the observed variable and $\bm R$ as the latent variable). Using the Gaussian assumption, the joint PDF $p(\boldsymbol R,\boldsymbol Y|\boldsymbol \theta)$ becomes
\begin{align*}
    p(\boldsymbol R,\boldsymbol Y|\boldsymbol \theta) = \tfrac{\mathbb I_{Q(\boldsymbol R)}(\boldsymbol Y)}{\pi\sigma^2_{\alpha}}e^{-\tfrac{(\boldsymbol R - \mathcal{H}_{cir}\boldsymbol \theta)^2}{\sigma^2_{\alpha}}}%
\end{align*}
with $\mathbb I_{\mathbb Q(\boldsymbol R)}(\boldsymbol Y)$ being one when $Q(\boldsymbol R)=\boldsymbol  Y$ and $(.)^2$ denoting element-wise square operator. Therefore $\hat P(\boldsymbol R)  := \hat P(\boldsymbol R|\boldsymbol Y,\boldsymbol \theta)$ in Block EM becomes
\begin{align}\label{latent_dist}
    \hat P(\boldsymbol R|\boldsymbol Y,\boldsymbol \theta) = \frac{e^{-\frac{(\boldsymbol R - \mathcal{H}_{cir}\boldsymbol \theta)^2}{\sigma^2_{\alpha}}}}{\int_\mathcal{D} e^{-\frac{(\boldsymbol R - \mathcal{H}_{cir}\boldsymbol \theta)^2}{\sigma^2_{\alpha}}}d\boldsymbol R}
\end{align}
where $\mathcal{D}_i = (-\infty,0)$ when  {$\bm Y_i<0$} and $\mathcal{D}_i = (0,+\infty)$ when  {$\bm{Y}_i>0$}. 
The second update of Block EM has the objective $\mathbb E_{\hat P(\bm R| {\bm{Y}},\bm{\theta})}\left[-\log(P(\bm R, {\bm{Y}}| \theta_{i_k}))\right]$. 
From the structure of $-\log(P(\bm R, {\bm{Y}}| \theta_{i_k}))$ we know
\[-\log(P(\bm R, {\bm{Y}}| \theta_{i_k})) = \frac{\bm R^T \bm R -\bm R^T\mathcal{H}_{cir}\bm{\theta} - \bm{\theta}^T\mathcal{H}_{cir}^T\bm R + \bm{\theta}^T\mathcal{H}_{cir}^T\mathcal{H}_{cir}\bm{\theta}}{\sigma^2_{\alpha}}.\]
Note that for simplicity, we considered $\theta_{i_k}$ in $\bm{\theta}$ as the optimization variable. Therefore, the first term is independent from the optimization variable and we replace it with $\hat{\bm R}^T \hat{\bm R} = \mathbb{E}_{\hat P(\bm R| {\bm{Y}},\bm\theta)}[\bm R]^T\mathbb{E}_{\hat P(\bm R| {\bm{Y}},\bm\theta)}[\bm R]$. Thus,
\begin{align}\label{eqn:conv_obj_MIMO}
\mathbb E_{\hat P(\bm R| { {\bm{Y}}},\bm{\theta})}\left[-\log(P(\bm R, {\bm{Y}}| \theta_{i_k}))\right]&= \frac{\hat{\bm R}^T \hat{\bm R} -\hat{\bm R}^T\mathcal{H}_{cir}\bm{\theta} - \bm{\theta}^T\mathcal{H}_{cir}^T\hat{\bm R} + \bm{\theta}^T\mathcal{H}_{cir}^T\mathcal{H}_{cir}\bm{\theta}}{\sigma^2_{\alpha}}\nonumber\\
&= \frac{\|\hat{\bm R} - \mathcal{H}_{cir}\bm{\theta}\|^2}{\sigma^2_{\alpha}}
\end{align}
Therefore, in this problem we can evaluate the expected value of the latent variables $\boldsymbol{R}$ through \eqref{latent_dist} and then minimize the convex objective \eqref{eqn:conv_obj_MIMO} to recover the transmitted symbols $\boldsymbol{\theta}$. The coordinate-wise calculation of the expected value of the latent variables $\boldsymbol{R}$ will take the form 
\begin{align}\label{coordinate_mimo_em1}
    \hat{\bm R}^{k+1} := \mathbb E_{\hat P(\bm R| {\bm{Y}},\bm{\theta}^k)}\left[\bm R\right] &= \sigma_{\alpha}^2\left(\frac{\bm Y_{Re}\phi(\bm Y_{Re}\bm Z_{Re}^k/\sigma_{\alpha})}{\Phi(\bm Y_{Re}\bm Z_{Re}^k/\sigma_{\alpha})} + j\frac{\bm Y_{Im}\phi(\bm Y_{Im}\bm Z_{Im}^k/\sigma_{\alpha})}{\Phi(\bm Y_{Im}\bm Z_{Im}^k/\sigma_{\alpha})}\right) \nonumber\\
    &+ \mathcal{H}_{cir}(:,i_k)\bm{\theta}_{i_k}^k 
\end{align}
where subscripts $Re,Im$ specify the real and the imaginary parts, $j=\sqrt{-1}$, $\bm Z^k=\mathcal{H}_{cir}\bm{\theta}^k$, $\phi(x) = \tfrac{1}{\sqrt{2\pi}}exp\{-\tfrac{x^2}{2}\}$, and $\Phi(x) = \int_{\infty}^x \phi(x')dx'$. With \eqref{coordinate_mimo_em1} at hand, we can easily update $\bm{\theta}^{k+1}_{i_k}$ as
\begin{align}\label{coordinate_mimo_em2}
    \bm{\theta}^{k+1}_{i_k} = \argmin_{\underset{\bm{x} \in \bm \{-3,-1,+1,+3\}}{\bm{x}\in \mathbb C^{\|i_k\|_0}}} \|\hat{\bm R}^{k+1}  - \mathcal{H}_{cir}(:,i_k)\bm{x}\|_2^2.
\end{align}
Therefore, iterative updates of \eqref{coordinate_mimo_em1} and \eqref{coordinate_mimo_em2} will recover $ \bm{\theta}^{k+1}_{i_k}$. Note that to project to $\{-3,-1,+1,+3\}$ we need to relax the constraint set to $x\in{[-3,3]}$, and assure the convergence of the updates before projection. This will lead to a small inner loop for each $i_k$.

\subsection{Numerical Experiments}
We conducted numerical simulations to verify the performance of the \algoem in this problem. 
Specifically, we are aiming to study the Bit Error Rate (BER) for various bit energy to noise spectral density ($E_b/N$) define {d} as
\begin{equation*}
    \frac{E_b}{N} = \frac{P \text{Tr}\left(\mathbb{E}_{\mathcal{H}_{cir}}\left[  \mathcal{H}_{cir}\mathcal{H}_{cir}^H\right]\right)}{MN\sigma_{\alpha}^2B},
\end{equation*}
 where $P$ denotes the total transmitted power and $B$ is number of bits in the signal constellation symbol.
 Consider a MIMO system employing 16-QAM with cyclic prefix omitted at the transmitter, $N=6$ users, and $M=64$ receiving antennas. The length of signals is $W=32$. 
 Channel was considered with complex Gaussian noise with mean zero and covariance matrix of $\sigma^2 \bm I$. The length of the fading is $L=5$ and signals undergo 4 paths to reach the receiver. 
 The number of inner loops for projection was 10 and the stopping criteria was $\|\xi^{u} -  \xi^{u-1}\|_2 \leq \lambda \|\xi^{u-1} \|_2 $ with $\lambda=10^{-3}$, $\xi^{u}= \frac{1}{25}\sum_{p=25(u-1)+1}^{25u} \bm{\theta}^{p} $. 
 This choice of $\xi^{u}$ is quite flexible and any short-term criteria to understand the behaviour of the update sequence can be applied. 
 The maximum number of iterations for the \algoem algorithm was $10^4$ and $10^3$ for the EM algorithm. 
 The algorithms were initialized at zero and the block size was 10. 
 The results are depicted in \Cref{fig:MIMO_bitsnr} for 50 Monte-Carlo simulations.
 
As shown in \Cref{fig:MIMO_bitsnr}, the \algoem performs slightly better than EM in this setting. Note that \algoem requires less computational resources than EM to maximize the likelihood function.

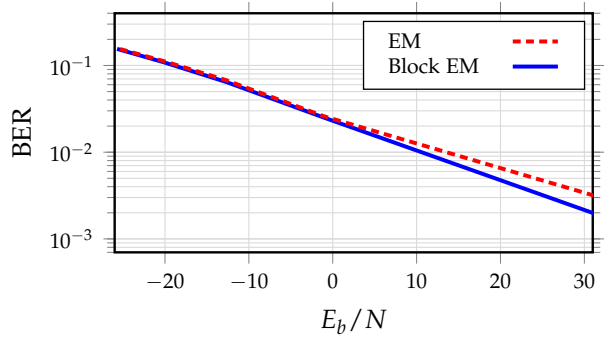
\begin{figure}[!t]
    \centering
    \begin{tikzpicture}
    
    \begin{groupplot}[group style={group size=1 by 1},        
        width = 0.50\textwidth, 
        height = 0.3\textwidth,        
        xmin = -26, 
        xmax = 31,      
        xlabel={$E_b/N$},    
        grid = both, 
        grid style = {gray!30},        
        tick label style = {font=\footnotesize}, 
        ]

    \nextgroupplot[
        ymin = 7e-4, 
        ymax = 4e-1,
        ymode = log,        
        ylabel = {BER},
        line width = 0.9pt,
        xtick align = outside,
        ytick align = outside,
        name = plot_mimo_bit,
        ]

    \addplot[blue, fill=none, name path=em_coo, line width=1.5pt] 
        coordinates {
(31.256602376952877,0.001953125000000) 
(-1.195146933571638,0.025338541666667)
(-12.702442388917058,0.064479166666667)
(-20.132264865503895,0.108776041666667)
(-25.714703572305243,0.155364583333333)}; 
        \label{bcem_mimo}
    \addplot[red,  fill=none, name path=em, line width=1.5pt, densely dashed] 
        coordinates {
(31.195464841815852,0.003151041666667)
(-0.984863269135394,0.025859375000000)
(-12.782174949724077,0.068020833333333)
(-20.146192984588428,0.112213541666667)
(-25.346415827374003,0.154687500000000)};
        \label{em_mimo}     
    \end{groupplot}
    \node[anchor=north east, draw = black, line width=0.5pt, fill=white, font=\footnotesize]  
        (legend) at ([shift={(-0.1cm,-0.1cm)}]plot_mimo_bit.north east) 
        {
            \begin{tabular}{l l}
                EM  & \ref*{em_mimo} \\
                \algoem & \ref*{bcem_mimo} 
            \end{tabular}
        };      
\end{tikzpicture}
\caption{Comparison of \algoem and classical EM in terms of BER vs. $E_b/N$.}
\label{fig:MIMO_bitsnr}
\end{figure} 

\end{document}